\documentclass[reqno,
12pt]{amsart}

\usepackage{
	amssymb,
	amsmath,
	amsthm,
	eucal,
	dsfont,
	multicol,
	mathrsfs,
	tikz,
	graphicx,
	hyperref
}
\usepackage{lipsum}
\makeatletter\renewcommand*{\eqref}[1]{%
	\hyperref[{#1}]{\textup{\tagform@{\!\!\ref*{#1}}}}%
}\makeatother 

\makeatletter
 
  \@addtoreset{equation}{section}
 \makeatother
\theoremstyle{plain}

\newtheorem{theorem}{Theorem}[section]
\newtheorem{lemma}[theorem]{Lemma}

\newtheorem{corollary}[theorem]{Corollary}
\theoremstyle{definition}
\newtheorem{definition}[theorem]{Definition}
\newtheorem{remark}[theorem]{Remark}

\newcommand{\bignorm}[1]{{\Big\|#1\Big\|}}
\newcommand{\norm}[1]{{\|#1\|}}

\def\supp{\mathop{\mathrm{supp}}\nolimits}

\def\Id{\mathop{\mathrm{Id}}\nolimits}

\def\Ker{\mathop{\mathrm{Ker}}\nolimits}

\def\Im{\mathop{\mathrm{Im}}\nolimits}

\def\sgn{\mathop{\mathrm{sgn}}\nolimits}
\def\R{{\mathbb{R}}}

\def\N{{\mathbb{N}}}
\def\C{{\mathbb{C}}}

\def\S{{\mathcal{S}}}

\def\H{{\mathcal{H}}}

\def\<{{\langle}}
\def\>{{\rangle}}

\DeclareMathOperator*{\slim}{s-lim}
\DeclareMathOperator*{\wlim}{w-lim}
\DeclareMathOperator*{\limi}{\underline{lim}}

\title[Higher-order Schr\"odinger operators with bound potentials]{Global Kato smoothing and Strichartz estimates for Schr\"odinger type equations with rough decay potentials}

\author{Haruya Mizutani}
\address{Department of Mathematics, Graduate School of Science, Osaka University, Toyonaka, Osaka 560-0043, Japan}
\email{haruya@math.sci.osaka-u.ac.jp}
\author{Xiaohua Yao\textsuperscript{\dag }}
\address{School of Mathematics and Statistics,   Key Laboratory of Nonlinear Analysis  and Applications (Ministry of Education), Central China Normal University, Wuhan, 430079, P.R. China}
\email{yaoxiaohua@ccnu.edu.cn}

\thanks{\textsuperscript{\dag} Corresponding author}

\begin{document}
\date{\today}

\begin{abstract}

	Let \( H = (-\Delta)^m + V \) be a higher-order elliptic operator on \( L^2(\mathbb{R}^n) \), where \( V \) is a general bounded decaying potential. This paper focuses on the global Kato smoothing and Strichartz estimates for  solutions to Schr\"odinger-type equation associated with \( H \). In particular, we first establish sharp global Kato smoothing estimates for \( e^{itH} \), based on uniform resolvent estimates of Kato-Yajima type for the absolutely continuous part of \( H \). As a consequence, we also obtain optimal local decay estimates. Using these local decay estimates, we then prove the full set of Strichartz estimates, including the endpoint case. Notably, we derive Strichartz estimates with sharp smoothing effects for higher-order cases with rough potentials, which are applicable to the study of nonlinear higher-order Schr\"odinger equations. Finally, we introduce new uniform Sobolev estimates of the Kenig-Ruiz-Sogge type, incorporating an additional derivative term, which are crucial for establishing the sharp Kato smoothing estimates.

\end{abstract}

\maketitle

\section{Introduction}
\subsection{Main results}

Consider the following higher-order elliptic operator with a potential on \( L^2(\mathbb{R}^n) \):
\[
H = H_0 + V(x), \quad H_0 = (-\Delta)^m,
\]
where \( n > 2m \), \( m \in \mathbb{N} \), and \( \Delta = \sum_{j=1}^n \partial_{x_j}^2 \) is the Laplacian. The potential \( V(x) \) is a real-valued measurable function that satisfies \( |V(x)| \leq C \langle x \rangle^{-s} \) for some \( s > 0 \), where \( \langle x \rangle = (1 + |x|^2)^{1/2} \).

It is well-known that higher-order elliptic operators (including more general forms \( P(D) + V \)) have been extensively studied as general Hamiltonian operators in various contexts. For instance, Schechter \cite{Schechter} has explored their spectral theory, while Kuroda \cite{Kur78}, Agmon \cite{Agmon} and H\"ormander \cite{Ho} have contributed to their scattering theory. Ben-Artzi and Devinatz \cite{Ben-Devi} investigated the limiting absorption principle, and Davies \cite{Da}, along with Davies and Hinz \cite{DaHi} and Deng et al. \cite{DDY}, have examined their semigroup theory. Additionally, Herbst and Skibsted \cite{HS15, HS17} studied the eigenfunctions of pseudodifferential operators (PDOs), and further interesting topics can be found in these works \cite{BS, SYY}.

For any bounded real potential \( V \), it follows directly from Kato-Rellich theorem (see, e.g., Simon \cite[Chapter 7]{Simon}) that \( H \) is a self-adjoint operator with the same Sobolev domain \( H^{2m}(\mathbb{R}^n) \) of order \( 2m \) as \( H_0 \). In particular, Stone's theorem asserts that the Schr\"odinger unitary group \( e^{itH} \) provides the solution to Schr\"odinger equation:
\begin{align}
	\label{Cauchy}
	(i\partial_t + H)\psi(t,x) = 0, \quad \psi(0,x) = \psi_0(x), \quad (t,x) \in \mathbb{R}^{1+n},
\end{align}
by \( \psi(t) = e^{itH}\psi_0 \) for \( t \in \mathbb{R} \).

In this paper, we aim to establish global regularity and decay estimates for the solution of the Cauchy problem \eqref{Cauchy}: global Kato smoothing estimates and  Strichartz estimates. These estimates are not only interesting in themselves, but also crucial for investigating the local and global behavior of nonlinear dispersive equations with potentials (see, e.g., \cite{MiWangYao}). This work extends our previous study \cite{MiYa1}, where we considered the case with  Hardy potential \( V(x) = a|x|^{-2m} \)  and more generally, critically decaying potentials \( V(x) = O(\langle x \rangle^{-2m}) \) under several repulsive conditions. It is well known that under such repulsive conditions, \( H \) is purely absolutely continuous and has no eigenvalues.

Here using a different approach based on the limiting absorption principle, we establish these estimates for a general bounded potential \( V(x) \) that decays slightly faster than \( |x|^{-2m} \) as \( |x| \to \infty \). Under this setting, \( H \) may possibly have negative eigenvalues. Such potentials naturally arise in many physically relevant models, such as one-body higher-order Schr\"odinger operators or linearized operators associated with soliton solutions to higher-order nonlinear Schr\"odinger equations.

 The first results are stated as follows:

\begin{theorem}
\label{Kato-smoothing}
Let $m\in  \N$, $m\ge2$, $n>2m$,  $H=(-\Delta)^m+V$ and $|V(x)|\le C \<x\>^{-s}$ for $s>2m$. Assume that $H$ has no positive eigenvalues and no zero resonance/eigenvalue (see Definition \ref{resonace} in Section \ref{local uniform resolvent estimates}). Let $P_{ac}(H)$ denote the projection onto the absolutely continuous subspace of $H$. Then the following statements
(i.e. \underline{the global kato smoothing estimates}) are satisfied:
\vskip0.2cm
 (i) If $m-n/2< \gamma< m-1/2$,  then
\begin{equation}
\label{eq1.1}
\big\||x|^{-m+\gamma}|D|^\gamma e^{itH}P_{ac}(H)\psi_0\big\|_{L^2_tL^2_x}\lesssim \norm{\psi_0}_{L^2_x}.\end{equation}
where $D=-(i\partial_{x_1},i\partial_{x_2},\cdots,i\partial_{x_n})$, $|D|=\sqrt{-\Delta}$. In particular, as $\gamma=0$, the following local decay estimate holds:
\begin{equation}
\label{eq1.2}
\big\||x|^{-m} e^{itH}P_{ac}(H)\psi_0\big\|_{L^2_tL^2_x}\lesssim \norm{\psi_0}_{L^2_x}.\end{equation}

(ii) If $\gamma=m-1/2$, then for any $\epsilon>0$,
\begin{equation}
\label{eq1.3}
\big\|\<x\>^{-1/2-\epsilon} |D|^{m-1/2}e^{itH}P_{ac}(H)\psi_0\big\|_{L^2_tL^2_x}\lesssim \norm{\psi_0}_{L^2_x}.\end{equation}
In particular, $e^{itH}\psi_0$ belongs to $\H^{m-1/2}_{Loc}(\R^n)$ for a.e. $t\in \R$ and initial data $\psi_0=P_{ac}(H)\psi_0\in L^2(\R^n)$, and satisfies \underline{the following local smoothing estimate}:
\begin{equation}
\label{eq1.4}\int_{-\infty}^\infty\int_{|x|\le R}\big||D|^{m-1/2}e^{itH}\psi_0\big|^2dxdt\le C(R)\|\psi_0\|_{L^2_x}^2.\end{equation}
\end{theorem}

\vskip0.4cm 
\begin{remark} Several remarks are given as follows:
	
\begin{itemize}

\item We first note that for \( m = 1 \), Kato \cite{K} showed the absence of positive eigenvalues for the operator \( H = -\Delta + V \) with potentials decaying as \( V = o(\langle x \rangle^{-1}) \) as \( |x| \to \infty \). However, for \( m \geq 2 \), the higher-order Schr\"odinger operators \( H = (-\Delta)^m + V \) may have discrete eigenvalues embedded in the positive real line, even for smooth potentials \( V \in C_0^\infty(\mathbb{R}^n) \) (see, e.g., Remark \ref{remark 3.1} below for the construction of counterexamples). This observation demonstrates that the decay and regularity of potentials do not always prevent the appearance of positive eigenvalues in higher-order cases.

On the other hand, it is known from the work of Feng et al. \cite{FSWY} that if a bounded potential satisfies the repulsive condition (i.e., \( (x \cdot \nabla)V \leq 0 \)), then for each \( m \geq 1 \), the operator \( H = (-\Delta)^m + V \) has no eigenvalues. This result is particularly useful for further studies of higher-order dispersive problems.
Additionally, for this and other assumptions given in Theorem \ref{Kato-smoothing}, more detailed comments can be found in Subsection \ref{fur-remark} and Remark \ref{remark 3.1} below.

\item Furthermore, we would like to provide additional remarks on the Kato smoothing estimates discussed above. When \( V = 0 \) (i.e., \( H = (-\Delta)^m \)), local smoothing estimates like \eqref{eq1.4}, originally traced back to Kato \cite{Ka} for the KdV equation, were first proved by Constantin and Saut \cite{Con-Saut1} for general dispersive equations. These estimates were further studied by many other authors, including Ben-Artzi and Devinatz \cite{Ben-Dev}, and Kenig-Ponce-Vega \cite{KPV}.

The global-in-time smoothing estimates \eqref{eq1.1} and \eqref{eq1.3} were first established by Kato and Yajima \cite{KaYa} for \( m = 1 \) (i.e., the Laplacian operator \( -\Delta \)) using uniform resolvent estimates, based on the smooth perturbation method \cite{Kat}. These results were also reproved by Ben-Artzi and Klainerman \cite{Ben-Kla} using the spectral measure integral.  

For higher-order operators \( (-\Delta)^m \) with \( m \geq 2 \), and even for fractional operators \( (-\Delta)^\alpha \) with \( \alpha > 0 \), optimal global smoothing estimates such as \eqref{eq1.1} and \eqref{eq1.3} have been studied in the works of Ruzhansky and Sugimoto \cite{RuSu}.  References therein provide further insights into these results.

\item

In the case where \( V \neq 0 \), local smoothing estimates have been considered by Constantin and Saut \cite{Con-Saut2} for (higher-order) Schr\"odinger equations, even under general perturbations of the form \( V(x,D) \). In particular, they proved the following local-in-time and space smoothing estimate (see Corollary 2.4 of \cite{Con-Saut2}):
\begin{equation}
	\label{eq1.5}
	\int_{-T}^T \int_{|x| \leq R} \big| |D|^{m - 1/2} \psi(t,x) \big|^2 dx \, dt \leq C(T,R) \|\psi_0\|_{L^2_x}^2,
\end{equation}
which is weaker than the estimate \eqref{eq1.4} since \( C(T,R) \) depends on \( T \).

For global-in-time smoothing estimates with potentials, the situation is more delicate than in the local-in-time case \eqref{eq1.5}, as it particularly depends on the global time and space behavior of the solution \( e^{itH}\psi_0 \) and the spectral properties of \( H \) at the threshold (i.e., the critical values of the symbol \( P(\xi) \)). For \( H = -\Delta + V \) (i.e., \( m = 1 \)) under the same conditions as in Theorem \ref{Kato-smoothing}, Ben-Artzi and Klainerman \cite{Ben-Kla} obtained the following (similar but not identical) global estimate:
\begin{equation}
	\label{eq1.6}
	\int_{-\infty}^\infty \int_{\mathbb{R}^n} \big|\langle x \rangle^{-1/2 - \epsilon} (1 + H)^{1/4} e^{itH} P_{ac}(H) \psi_0 \big|^2 dx \, dt \leq C \|\psi_0\|_{L^2_x}^2,
\end{equation}
(see also \cite[Theorem 1.10]{Miz2}, where the same estimate as \eqref{eq1.3} was obtained for \( m = 1 \)).

To the best of our knowledge, global smoothing estimates like \eqref{eq1.1} and \eqref{eq1.3} are less known for \( m \geq 2 \). In particular, we remark that the estimates \eqref{eq1.1} and \eqref{eq1.3} in Theorem \ref{Kato-smoothing} are optimal when compared to the free case (see, e.g., Ruzhansky and Sugimoto \cite{RuSu}). Moreover, the local decay estimates \eqref{eq1.2} will play a crucial role in establishing the endpoint Strichartz estimates for equation \eqref{eq1.1} (see Theorem \ref{theorem_4_1} below).

\end{itemize}
\end{remark}

Besides, in the following Theorem \ref{theorem_2},  by an abstract argument due to \cite{Kat} and \cite{DAn}, we are also able to establish Kato smoothing effect for the solution $\psi$ of the inhomogeneous equation
\begin{align}
\label{Cauchy_2}
(i\partial_t+H) \psi(t,x)=F(t,x),\quad \psi(0,x)=\psi_0(x),
\end{align}
with data $\psi_0\in \S(\R^n)$ and $F\in \S(\R\times\R^n)$ (the Schwartz function space).  The solution $\psi$ is  given by the Duhamel formula:
\begin{equation}
\label{Duhamel}
\psi=e^{itH}\psi_0+i\int_0^te^{i(t-s)H}F(s)ds.
\end{equation}
\vskip0.3cm
\begin{theorem}	
\label{theorem_2}
Under the same conditions as in Theorem \ref{Kato-smoothing}, the solution $\psi$ to \eqref{Cauchy_2} given by \eqref{Duhamel} satisfies the following statements:
\vskip0.3cm
(i) If $m-n/2< \gamma< m-1/2$, then
\begin{align*}
\big\||x|^{-m+\gamma}|D|^{\gamma}P_{ac}(H) \psi\big\|_{L^2_tL^2_x}&\lesssim \norm{\psi_0}_{L^2_x}+\big\||x|^{m-\gamma}|D|^{-\gamma}F\big\|_{L^2_tL^2_x}.
\end{align*}
Furthermore, if $\gamma=0$, then the following local decay estimate holds:
\begin{align*}
\big\||x|^{-m}P_{ac}(H)\psi\big\|_{L^2_tL^2_x}&\lesssim \norm{\psi_0}_{L^2_x}+\big\||x|^m F\big\|_{L^2_tL^2_x}.
\end{align*}

(ii) If $\gamma= m-1/2$, then for any $\epsilon>0$,
\begin{align*}
\big\|\<x\>^{-1/2-\epsilon} |D|^{m-1/2}P_{ac}(H)\psi\big\|_{L^2_tL^2_x}&\lesssim \norm{\psi_0}_{L^2_x}+\big\|\<x\>^{1/2+\epsilon} |D|^{-m+1/2}F\big\|_{L^2_tL^2_x}.
\end{align*}
\end{theorem}

\vskip0.4cm 

Finally, we come to state the results on the Strichartz estimates for the equation \eqref{Cauchy_2} above under the same conditions as in Theorem \ref{Kato-smoothing}, and  will see that the $L^p$-type of global smoothing estimates can happen for higher-order dispersive equations. To the end,  recall that $(1/p,1/q)\in [0,1]^2$ is said to be a (sharp) $\alpha$-admissible pair if
\begin{align}
\label{admissible}
2\le p,q\le \infty,\quad 1/p=\alpha(1/2-1/q),\quad (p,q,\alpha)\neq (2,\infty,1).
\end{align}

\begin{theorem}
\label{theorem_4_1}
Under the same conditions as in Theorem \ref{Kato-smoothing}, the solution $\psi$ to \eqref{Cauchy_2} satisfies the following statements:
 \vskip0.3cm
(i) If $(p_1,q_1)$ and $(p_2, q_2)$ satisfy \eqref{admissible} with $\alpha=n/(2m)$,  then  $\psi$ satisfies
the following standard Strichartz estimates:
\begin{align}
\label{eq4.4}
\big\|P_{ac}(H)\psi\big\|_{L^{p_1}_tL^{q_1}_x}\lesssim \|\psi_0\|_{L^2_x}+\|F\|_{L^{p_2'}_tL^{ q_2'}_x}.
\end{align}
In particular,  the following endpoint Strichartz estimates hold:
\begin{align}
\label{eq4.4'}
\big\|e^{itH}P_{ac}(H)\psi_0\big\|_{L^{2}_tL^{\frac{2n}{n-2m}}_x}&\lesssim \|\psi_0\|_{L^2_x},\\
\label{eq4.4''}
\bignorm{\int _{0}^t e^{i(t-s)H}P_{ac}(H)F(s)ds}_{L^2_xL^{\frac{2n}{n-2m}}_x}&\lesssim \|F\|_{L^2_xL^{\frac{2n}{n+2m}}_x}.
\end{align}
 \vskip0.3cm
(ii) Let $(p_1,q_1)$ and $(p_2,q_2)$ satisfy \eqref{admissible} with $\alpha=n/2$. Then $\psi$ satisfies
 the following improved Strichartz estimates with gain of regularity:
\begin{align}
\label{eq4.3}
\big\||D|^{2(m-1)/p_1}P_{ac}(H)\psi\big\|_{L^{p_1}_tL^{q_1}_x}\lesssim \|\psi_0\|_{L^2_x}+\big\||D|^{2(1-m)/p_2}F\big\|_{L^{p_2'}_tL^{q_2'}_x}.
\end{align}
\vskip0.3cm
(iii) Let  $\psi_0=P_{ac}(H)\psi_0\in L^2(\R^n)$, $F\equiv 0$ and $(p_1,q_1)=(2,2n/(n-2))$. Then the estimate \eqref{eq4.3} implies that $\psi=e^{itH}\psi_0$ belongs to $\H_{Loc}^{m-1}$ for a.e. $t\in \R$, and satisfies \underline{ the following global $L^p$-smoothing estimate} (comparing with the estimate \eqref{eq1.4}):
\begin{equation}
\label{eq4.5'}
\int_{-\infty}^\infty\Big(\int_{\R^n}\big||D|^{m-1}e^{itH}\psi_0\big|^{\frac{2n}{n-2}}dx\Big)^{\frac{n-2}{n}}dt\le C \ \|\psi_0\|_{L^2_x}^2.\end{equation}
\end{theorem}

\vskip0.4cm

Note that when $m=1$, the inequalities \eqref{eq4.4} and  \eqref{eq4.3} are the same and have been proved by Rodnianski-Schlag \cite{RoSc}. However, for $m>1$, the situations are different. In fact, if $(p,q)$ is $n/2$-admissible and $(p,q_1)$ is $n/(2m)$-admissible, then $p,q,q_1$ must satisfy $1/q-1/q_1=2(m-1)/(np)$. Sobolev's inequality then implies
\begin{align}
\label{Sobolev_2}
\norm{f}_{L^{q_1}}\lesssim \norm{|D|^{{2(m-1)}/{p}}f}_{L^{q}}.
\end{align}
Hence \eqref{eq4.4} immediately follows from \eqref{eq4.3}. In this sense, \eqref{eq4.3} has an additional smoothing effect compared with \eqref{eq4.4} in the higher-order case $m>1$. 

Moreover, we notice that such  Strichartz estimates with gain of regularity above have been extensively studied for the free case $H=(-\Delta)^m$ for all $m>1$ and played an important role in the study of higher-order and fractional NLS equations (see e.g. \cite{KPV}, \cite{Pau} and references therein). Therefore, we believe that Theorems \ref{theorem_4_1} have many potential applications to higher-order NLS equations with potentials satisfying $V(x)=O(\<x\>^{-2m-\epsilon})$.

Finally, we should mention that,  besides of Kato smoothing and Strichartz estimates above,  there are some recent interesting works on the pointwise time-decay estimates of $e^{itH}$ for higher-order Schr\"odinger operators with fast decay potentials. In case of $H=(-\Delta)^m+V$, for instance, we can refer to \cite{FSY}, \cite{GT} and \cite{Erdogan-Green-Toprak} for Kato-Jensen estimates or $L^1$-$L^\infty$ estimates  of $e^{itH}$ in the case $m=2$, and \cite{FSWY} for general $m\ge2$. However, comparing with the abundant point-wise decay results on  Schr\"odinger operators $-\Delta+V$ (see e.g. Journ\'{e}  et al \cite{JSS}, Yajima \cite{Yajima-JMSJ-95}, Schlag's survey \cite{Sch} and references therein), the $L^1$-$L^\infty$ decay estimates for higher-order Schr\"odinger operators are still far from completion and  deserve to be further investigated.

\subsection{Further remarks}
\label{fur-remark}We will make  more remarks on all theorems above. In particular, the optimality of the conditions for the operator $(-\Delta)^{m}+V(x)$ is discussed.
\begin{itemize}

\item ({\bf The restriction of dimension $n>2m$})  The dimensional restriction \( n > 2m \) plays a critical role in the validity of several important estimates, including the local decay estimates with \(\gamma = 0\) (as in equation (1.2)), uniform Sobolev estimates in Theorem \ref{Sobolev estimates}, and the endpoint Strichartz estimates in Theorem \ref{theorem_4_1} when either \( p_1 = 2 \) or \( p_2 = 2 \). This restriction ensures the necessary spatial and decay conditions for these estimates to hold.

For instance, without the condition \( n > 2m \), the estimates do not hold even for the free operator \( H = H_0 \). The failure of these estimates for the free operator implies that the methods used to establish the estimates for more general operators would also break down based on  the perturbation argument used in Sections below..

\vskip0.2cm
\item ({\bf The decay rate of potential $V$}) 
The decay index \( s > 2m \) is indeed critical for the validity of several estimates, particularly when dealing with Schr\"odinger operators (i.e., \( m = 1 \)). 
For example, in the case of Schr\"odinger operators (\( m = 1 \)), the work of Goldberg, Vega, and Visciglia \cite{GVV} demonstrated that Strichartz estimates fail to hold for certain classes of repulsive potentials that decay slower than \( |x|^{-2} \). This shows that for potentials with slower decay, even fundamental estimates as Strichartz estimates collapse, leaving only the trivial \( L^\infty_tL^2_x \) estimate. This highlights the delicate balance between the decay of the potential and the validity of key dispersive and smoothing estimates.

In contrast, in our previous paper \cite{MiYa1}, we considered higher-order and fractional Schr\"odinger operators with scaling-critical potentials of the form \( V(x) = O(\langle x \rangle^{-2m}) \), alongside certain conditions on \( \nabla V \). In that work, we successfully established Kato smoothing and Strichartz estimates for higher-order Schr\"odinger operators with critical-decay Hardy-type potentials of the form \( a|x|^{-2m} \). However, the approach in \cite{MiYa1} was based on Mourre theory, which is fundamentally different from the method used in the current paper.
It's important to note that the potential \( V(x) \) considered in this work generally does not satisfy the assumptions in \cite{MiYa1}, which is why the two results do not overlap.

\vskip0.2cm
\item  ({\bf The absence of positive eigenvalues of $H$})
The assumption regarding the absence of positive eigenvalues is indeed crucial for the dispersive analysis of the time-evolution operator \( e^{itH} \). For the case of \( m = 1 \), corresponding to the Schr\"odinger operator \( H = -\Delta + V \), Kato \cite{K} established that if the potential decays as \( |V(x)| = o(|x|^{-1}) \) as \( |x| \to \infty \), then the operator has no positive eigenvalues embedded in the continuous spectrum. This foundational result was later generalized by Agmon \cite{Agmon1}, Simon \cite{Simon3}, and Froese et al. \cite{FHHH},   Ionescu and Jerison \cite{IJ} and Koch and Tataru \cite{KoTa}.

However, for higher-order operators of the form \( H = (-\Delta)^m + V \) with \( m \geq 2 \), the situation is more complicated. In fact, for any even \( m = 2k \) (with \( k \in \mathbb{N} \)), it has been shown that there exist compactly supported smooth potentials \( V \) such that \( H = (-\Delta)^m + V \) can have positive eigenvalues.

Despite this, there are general virial criteria that provide conditions for the absence of positive eigenvalues for a wider class of higher-order (and even fractional) operators. These criteria often apply to repulsive potentials that satisfy conditions \( x \cdot \nabla V \leq 0 \).

\end{itemize}

\subsection{Organization of the paper}
In this subsection, we will describe the organization of paper and outline some specific results and ideas.

In Section \ref{section 2},   the limiting absorption principle of resolvent $R_0(z)$ of $H_0=(-\Delta)^m$  was first considered, which says that the resolvent operator $R_0(z)$ can continue into the spectrum point $\lambda\ge 0$ if one considers $R_0(z)$ as an operator function with value in $\mathbb B(L^2_s, L^2_{-s})$ for some $s>0$. For any $\lambda>0$ (regular values), we need only choose $s>1/2$ due to Agmon \cite{Agmon}. For $\lambda=0$ (threshold value), the limit is more complex than regular values since the {\it zero is the critical point} ( i.e. $\nabla |\xi|^{2m}=0 $ when $\xi=0$ ), we need to take $s>m$ for $\lambda=0$.

  Secondly, we also show the following  uniform Sobolev estimates in the sense of Kenig-Ruiz-Sogge type:
\begin{equation*}
\norm{|D|^{\alpha} u}_{L^{q}}\lesssim \norm{((-\Delta)^m-z)u}_{L^{p}},\ z\in\C, \ u\in C_0^\infty(\R^n),
\end{equation*}
where $2m-n<\alpha\le 2m-2n/(n+1)$, $1/p-1/q={(2m-\alpha)/n}$, $1<p<2n/(n+1)$ and $2n/(n-1)<q<\infty$. 

Note that when $2m-n<0$ and  $\alpha=0$, the uniform estimates have been proved by  Kenig-Ruiz-Sogge \cite{KRS} for $m=1$ and Huang-Yao-Zheng \cite{HYZ} for $m\ge2$.  When $\alpha\neq 0$, these estimates with derivatives are new  for all $m\ge 2$, which are indispensable to show  the sharp Kato smoothing estimates of $e^{itH}$ with potentials shown in Section \ref{local uniform resolvent estimates} below.

In Section \ref{section 3}, we will prove Theorems \ref{Kato-smoothing} and \ref{theorem_2}. The main point of the proofs is to establish the following uniform resolvent estimates of Kato-Yajima type (i.e. $H$-supersmoothing estimates):
\begin{align*}
\sup_{\lambda\in \R\,\theta\in (0,1]}\big\||x|^{-m+\gamma}|D|^\gamma P_{ac}(H)(H-\lambda\mp i\theta)^{-1}|D|^\gamma |x|^{-m+\gamma}\big\|_{L^2-L^2}<\infty,
\end{align*}
and
\begin{align*}
\sup_{\lambda\in \R\,\theta\in (0,1]}\big\|\<x\>^{-1/2-\epsilon}|D|^{m-1/2}P_{ac}(H)(H-\lambda\mp i\theta)^{-1}|D|^{m-1/2}\<x\>^{-1/2-\epsilon}\big\|_{L^2-L^2}<\infty,
\end{align*}
which in turn is based on the following resolvent formula:
\begin{equation*}
		R_V(z)=R_0(z)-R_0(z)v\big[M(z)\big]^{-1}vR_0(z).
	\end{equation*}
Here $v=\sqrt{|V|}$, $U=\sgn V(x)$ and  $M(z)=U+vR_0(z)v$ will be proved to have a uniformly bounded inverse on $L^{2}(\R^n)$ at the absolutely continuous spectra regime of $H$.

Section \ref{section 4} is devoted to proving Strichartz estimates in Theorem \ref{theorem_4_1}. The proof of Theorem \ref{theorem_4_1} relies on a perturbation method by Rodnianski-Schlag \cite{RoSc} (see also  \cite{BPST2}, \cite{BoMi}). Specifically,  let  $U_H,\Gamma_H$ be homogeneous and inhomogeneous Schr\"odinger evolutions defined by
\begin{align*}
U_Hf=e^{itH}f,\ \Gamma_HF=\int_0^t e^{i(t-s)H}F(s)ds.
\end{align*}
Then they satisfy the following Duhamel formulas
$$
U_H=U_{H_0}+i\Gamma_{H_0}VU_H,\quad \Gamma_H=\Gamma_{H_0}+i\Gamma_{H_0}V\Gamma_{H}=\Gamma_{H_0}+i\Gamma_{H}V\Gamma_{H_0}.
$$
Using these formulas, Sobolev's inequality \eqref{Sobolev_2}, \eqref{eq1.2} in Theorem \ref{Kato-smoothing}  and the same Strichartz estimates as \eqref{lemma_4_2} for $U_{H_0}$ and $\Gamma_{H_0}$, we can finish the proof of Theorem \ref{theorem_4_1}.

In Section \ref{section 5}, we will give the proofs of Theorems \ref{LIP-free case} and \ref{Sobolev estimates} based on many specific presentations of the free resolvent $R_0(z)=((-\Delta)^m-z)^{-1}$. For instance, in order to deal with limiting absorption principle at the zero energy $z=0$ in Theorems \ref{LIP-free case}, the following formula
\begin{equation*}
R_{0}(z)=\frac{1}{mz}\sum_{\ell=0}^{m-1}z_\ell\big(-\Delta-z_\ell\big)^{-1}, \ z_\ell=z^{\frac{1}{m}}e^{i\frac{2\ell\pi}{m}},\ z\in \C\setminus[0,\infty),
\end{equation*}
will be used. Moreover, in the proof of Theorem \ref{Sobolev estimates}, we will use oscillatory integral techniques from Harmonic analysis based on the asymptotic properties of the Fourier transform of the kernel of $R_0(z)$.  These details will be left to the last section.

\subsection{Notations}
In order to state our main results, we will use the following notations.
\begin{itemize}
\item $\<x\>$ stands for $\sqrt{1+|x|^2}$.
\item Let $\C^{\pm}=\{z\in\C; \ \pm\Im z>0\}$ denote the upper and lower complex planes, respectively, and $\overline{\C^{\pm}}$ be the closures of $\C^{\pm}$.
\item Let $\mathbb B(X,Y)$ be the Banach space of bounded operators from $X$ to $Y$,  $\mathbb B(X)=\mathbb B(X,X)$ and $\norm{\cdot}_{X\to Y}:=\norm{\cdot}_{\mathbb B(X,Y)}$. Let $\mathbb B_\infty(X)$ be  the set of compact operators on $X$. 
\item For $p\in [1,\infty]$, $p':=p/(p-1)$ is its H\"older conjugate exponent.
\item For each $s\in\R$, $\H^s(\R^n)$ denotes the $L^2$-based Sobolev space of order $s$ and $L^2_{s}(\R^n)$ denotes the usual weighted space consisting of the function $f$ with $\<x\>^sf\in  L^2(\R^n)$.
\item $\<f,g\>$ denotes the inner product $\int_{\R^n} f(x)\overline{ g(x)}dx$ in $L^2(\R^n)$, as well as the duality couplings $\<\cdot,\cdot\>_{L^{p'},L^{p}}$ and $\<\cdot,\cdot\>_{\H^{-m},\H^m}$.
\item $L^p_tL^q_x:=L^p(\R;L^q(\R^n))$ and $L^p_TL^q_x:= L^p([-T,T];L^q_x(\R^n))$.

\item For positive constants $A,B$, $A\lesssim B$ (resp. $A\gtrsim B$) means that  $A\le cB$ with some constant $c>0$ (resp, $A\ge cB$). $A\sim B$ means $cB\le A\le c'B$ with some $0<c<c'$.

\end{itemize}

\section{Free resolvent estimates for $H_0=(-\Delta)^m$}
\label{section 2}

\subsection{The limiting absorption principle for $R_0(z)$}
Let $H_{0}=(-\Delta)^m$ be the polyharmonic operator on $L^2(\R^n)$, where $m\in \N$, $m\ge2$, $n>2m$ and $\Delta=\sum_{j=1}^n\partial_{x_j}^2$ is the Laplacian.
It is well-known that $H_0$ is self-adjoint on $L^2(\R^n)$ with the domain $\mathcal{D}(H_0)=\H^{2m}(\R^n)$ and the spectrum of $H_0$ is $[0,\infty)$ by the Fourier transform presentation $\widehat{H_0f}(\xi)=|\xi|^{2m}\hat{f}(\xi)$.

For any $z\in \C\setminus[0,\infty)$, the resolvent $R_0(z)=(H_0-z)^{-1}$ is well-defined as a bounded operator on $L^2(\R^n)$  and its operator norm is $$\|(H_0-z)^{-1}\|_{L^2-L^2}=d(z, [0,\infty))^{-1}.$$
Moreover,  the resolvent $R_0(z)$ is analytic on $ z\in\C\setminus[0,\infty)$ in the uniform operator topology of $\mathbb B(L^2(\R^n))$. As $z$ closes to $\lambda\ge0$, it is clear that $R_0(z)$ can not continue into the spectrum point $\lambda\ge 0$ in the uniform operator topology of $\mathbb B(L^2(\R^n))$ (or any weak $L^2$ topology).  The celebrated limiting absorption principle, however, shows that such limits do exist if one considers $R_0(z)$ as an operator function with value in $\mathbb B(L^2_s, L^2_{-s})$ for some $s>0$. Indeed, for any $\lambda>0$ (regular values), we need only choose $s>1/2$. 

For $\lambda=0$ (the threshold value), the limit is more complex than the case with regular values since the zero is the critical point (i.e. $\nabla |\xi|^{2m}=0 $ when $\xi=0$). The studying situation depends on the specific operator.  In the present case (i.e. $H_0=(-\Delta)^m$), we need to take $s>m$ for $\lambda=0$.

In the following two theorems, we will present the limiting absorption principle of $R_0(z)$ and collect several interesting uniform resolvent estimates. The proofs of Theorems \ref{LIP-free case} and \ref{Sobolev estimates} will be given in Section 5 below.

\begin{theorem}\label{LIP-free case}
Let $n>2m$, $m\ge2$ and $H_{0}=(-\Delta)^m$. Consider $R_0(z)=(H_0-z)^{-1}$ as an analytic operator function on $\C\setminus [0,\infty)$ with values in in $\mathbb B(L^2_s, L^2_{-s})$ for some $s>0$. Then the following conclusions hold.
\vskip0.2cm
{\rm (i)} If $s>1/2$, then the following two limits exist for any $\lambda>0$ in the uniform operator topology of $\mathbb B(L^2_s, L^2_{-s})$:
\begin{equation}\label{eq2.1}
\lim_{\C^\pm \ni z\rightarrow \lambda} R_0(z)=R_0^\pm(\lambda),
\end{equation}
where $R_0^\pm(\lambda)$  can be written as follows:\begin{equation}\label{eq2.4}
\<R^\pm_0(\lambda)f,g\>={\rm p.v.}\int_{\R^n}\frac{\hat{f}(\xi)\overline{\hat{g}}(\xi)}{|\xi|^{2m}-\lambda}d\xi\pm \frac{1}{2m}\lambda^{1-2m\over 2m}\pi i\int_{|\xi|=\lambda^{1\over2m}}\hat{f}(\xi)\overline{\hat{g}}(\xi)d\sigma(\xi),
\end{equation}
for $f,\ g \in C_0^\infty(\R^n)$.
In other words, $R_0(z)$ is a uniform-operator-topology continuous function on $\overline{\C^{\pm}}\setminus\{0\}$ with values in $\mathbb B(L^2_s, L^2_{-s})$ for $s>1/2$.
 \vskip0.2cm
{\rm (ii)} If $s>m$, then $R_0(z)$ is an uniform-operator-topology continuous function on $\overline{\C^{+}}$  (also on $\overline{\C^{-}}$) with values in $\mathbb B(L^2_s, L^2_{-s})$, where  we take $R_0(0)=(-\Delta)^{-2m}$ and $R_0(z)=R^{\pm}_0(\lambda)$ for $z=\lambda\in \overline{\C^{\pm}}$ and $\lambda>0$.
\vskip0.2cm
{\rm (iii)} The following uniform estimates for $R_0(z)$ hold:
\begin{equation}\label{eq2.2}
\sup_{z\in\C\setminus[0,\infty)}\big\|R_0(z)\big\|_{L^2_s-L^2_{-s}}\le C_{s,m}<\infty,  \ \ s>m;
\end{equation}
\begin{equation}\label{eq2.3}
\sup_{z\in\C\setminus[0,\infty)}\big\||D|^{m-{1\over 2}}R_0(z)|D|^{m-{1\over 2}}\big\|_{L^2_s-L^2_{-s}}\le C'_{s,m}<\infty,  \ \ s>1/2.
\end{equation}
where $|D|=\sqrt{-\Delta}$, $C_{s,m}$ and $C'_{s,m}$ are the two positive constants independent of $z$.
\end{theorem}

\begin{remark}
\label{remark 2.1}For Theorem \ref{LIP-free case} above, we  make several comments as follows:

\begin{itemize}

\item The statement (i) is actually well-known due to Agmon \cite{Agmon}, where general higher-order elliptic operators $P(D)$ or the differential operators of principal type were considered.  Note that by Stone's formula, we have
 \begin{equation}\label{eq2.4'}\<E'_{H_0}(\lambda)f,g\>=\frac{1}{2\pi i}\<(R^+_0(\lambda)-R^-(\lambda))f,g\>,\ \ \lambda>0,\end{equation}
where $E'(\lambda)$ is the spectral measure density of $H_0$. Hence for any $s>1/2$,  it follows by the trace lemma that (see e.g. \cite{Ben-Kla}):
\begin{equation}\label{eq2.4''}
|\<E'_{H_0}(\lambda)f,f\>|=\frac{1}{2m}\lambda^{{1-2m\over 2m}}\int_{|\xi|=\lambda^{1\over2m}}|\hat{f}(\xi)|^2d\sigma\lesssim\min(\lambda^{{2s-2m\over 2m}}, \lambda^{{1-2m\over 2m}})\|f\|^2_{L^2_s}.
\end{equation}
\vskip0.2cm
\item In the statement (ii), the main point is to prove that $R_0(z)$ is continuous at $z=0$ in the cost of higher weight index. We will adopt the arguments from Agmon \cite{Agmon} to give the proof of (ii) in Section \ref{section 5}, where we use the following resolvent decomposition formula:
\begin{equation}\label{deompformula}
R_{0}(z)=\big((-\Delta)^{m}-z\big)^{-1}=\frac{1}{mz}\sum_{\ell=0}^{m-1}z_\ell\big(-\Delta-z_\ell\big)^{-1},
\end{equation}
where $z_\ell=z^{\frac{1}{m}}e^{i\frac{2\ell\pi}{m}}, z\in \C\setminus[0,\infty)$. This can make us to use the resolvent kernel of $-\Delta$ to present the kernel of $R_0(z)$ for $z\neq0$. Moreover, note that $\|(-\Delta-z)\|_{L^2_s-L^2_{-s}}=O(|z|^{-1/2})$ for any $s>1/2$ as $|z|\rightarrow \infty$ (see. e.g. \cite[p. 59]{KK12}), we then immediately deduce the following high energy decay estimates from the formula \eqref{deompformula}:
 \begin{equation}\label{highenergy}
\|R_{0}(z)\|_{L^2_s-L^2_{-s}}\le C_{s,m,\delta}\ |z|^{(1-2m)/2m},\ \ |z|\ge\delta>0,\ s>1/2,
\end{equation}
where $C_{s,m,\delta}$ is a positive constant depending on $s, m, \delta$.
\vskip0.2cm
\item In the statement (iii), the two uniform estimates \eqref{eq2.2} and \eqref{eq2.3}  play critical roles in the present paper. In fact, by the analytic interpolation, we can deduce from these estimates that $\<x\>^{-m-\epsilon+\gamma}|D|^{\gamma}$ is $H$-supersmooth in the sense of Kato-Yajima(see \cite{KaYa}):
    \begin{equation}\label{eq2.5}
\sup_{z\in\C\setminus[0,\infty)}\big\|\<x\>^{-m-\epsilon+\gamma}|D|^{\gamma}\ (H_0-z)^{-1}\ |D|^{\gamma}\<x\>^{-m-\epsilon+\gamma}\big\|_{L^2-L^2}<\infty,
\end{equation}
for any $0\le \gamma\le m-\frac{1}{2}$ and $\epsilon>0$, which immediately implies the local decay and Kato smoothing estimates of $e^{itH_0}$, that is,
\begin{equation}
\label{eq2.6}
\norm{\<x\>^{-m-\epsilon+\gamma}|D|^\gamma e^{itH_0}\psi_0}_{L^2_tL^2_x}\lesssim \norm{\psi_0}_{L^2_x}.\end{equation}
 In particular, as $\gamma=0$,  the inequality \eqref{eq2.6} becomes the local decay estimate.
\end{itemize}
\end{remark}

\begin{remark} The further remarks about the estimates \eqref{eq2.2} and \eqref{eq2.3} are given as follows:
\vskip0.2cm
\begin{itemize}

\item The inequality \eqref{eq2.2} can be deduced from the following uniform Sobolev estimate:
\begin{equation}
\label{eq2.7}
\norm{u}_{L^{\frac{2n}{n-2m}}}\lesssim \norm{(H_0-z)u}_{L^{\frac{2n}{n+2m}}},\ z\in\C, \ u\in C_0^\infty(\R^n),
\end{equation}
which was proved by Kenig-Ruiz-Sogge \cite{KRS} for $m=1$ and \cite{HYZ} for $m\ge2$. In fact, if $s>m$, then it follows by H\"older's inequality that
\begin{equation}
\label{eq2.8}
\norm{\<x\>^{-s} R_0(z)\< x\>^{-s}}_{L^2-L^2} \le \norm{\<x\>^{-s}}_{L^{n/m}}^2\norm{R_0(z)}_{L^{{2n\over n+2m}}-L^{{2n\over n-2m}}}<\infty
\end{equation}
uniformly in $z\in \C\setminus[0,\infty)$, which gives the desired estimate \eqref{eq2.2}. Moreover, by means of the real interpolation theory, the estimate \eqref{eq2.7} can be refined to the estimate
$$
\norm{u}_{L^{\frac{2n}{n-2m},2}}\lesssim \norm{(H_0-z)u}_{L^{\frac{2n}{n+2m},2}}
$$
where $L^{p,q}(\R^n)$ is the Lorentz space (see \cite{Gra}). Since $|x|^{-m}\in L^{n/m, \infty}$, by the weak H\"older inequality (see \cite{BeLo}) and \eqref{eq2.7},  we have the following sharp uniform estimates:
\begin{equation}
\label{eq2.8'}
\sup_{z\in \C\setminus[0,\infty)}\big\||x|^{-m} R_0(z)|x|^{-m}\big\|_{L^2-L^2} \lesssim \sup_{z\in \C\setminus[0,\infty)}\norm{R_0(z)}_{L^{{2n\over n+2m},2}-L^{{2n\over n-2m},2}}<\infty.
\end{equation}
\vskip0.2cm
\item  The inequality \eqref{eq2.3} is due to Agmon \cite[Lemma A.2]{Agmon}, where he actually showed that for any $s>1/2$ and $u\in C_0^\infty(\R^n)$,
\begin{equation}
\label{eq2.9}
\int (1+|x_j|^2)^{-s}|P^j(D)u|^2 dx \le C_{m, s}\int (1+|x_j|^2)^{s}|(P(D)u|^2dx,
\end{equation}
where $P(D)$ is any differential operator of order $2m$  and $P^j(\xi)=\frac{\partial}{\partial \xi_j}P(\xi)$ for $j=1,2,\cdots,n$. Taking $P(D)=(-\Delta)^{m}-z$, then the desired inequality \eqref{eq2.3} immediately follows from the \eqref{eq2.9} above.
\end{itemize}
\end{remark}

Besides of the special uniform Sobolev estimates \eqref{eq2.7}, we actually can show the following uniform Sobolev estimates with derivatives, which will play key roles in sharp Kato smoothing estimates of $e^{itH}$ with potentials shown in Section \ref{local uniform resolvent estimates}.

\begin{theorem}
\label{Sobolev estimates}
Let $H_0=(-\Delta)^m$ with $n>2m$ and $R_0(z)=(H_0-z)^{-1}$ for $z\in \C\setminus[0,\infty)$. Then for any $2m-n<\alpha\le 2m-2n/(n+1)$,  the following uniform estimates hold:
\begin{equation}
\label{eq2.10}
\norm{|D|^{\alpha} u}_{L^{q}}\lesssim \norm{(H_0-z)u}_{L^{p}},\ z\in\C, \ u\in C_0^\infty(\R^n),
\end{equation}
where $1/p-1/q={(2m-\alpha)/n}$, $1<p<2n/(n+1)$ and $2n/(n-1)<q<\infty$. Moreover, the following uniform $L^p$-$L^q$ resolvent estimates hold:
 \begin{equation}
\label{eq2.11}
\sup_{z\in \C\setminus[0,\infty)}\big\||D|^{\alpha/2} R_0(z)|D|^{\alpha/2}\big\|_{L^p-L^{q}}=\big\||D|^{\alpha} R_0(z)\big\|_{L^p-L^{q}}<\infty.
\end{equation}
\end{theorem}

\vskip0.4cm
\begin{remark}
\label{remark_1}
For any $\alpha,p,q$ satisfying the conditions in Theorem \ref{Sobolev estimates}, one can find $p_0,p_1,$ $q_0,q_1\in (1,\infty)$ satisfying  $p_0<p<p_1$, $q_0<q<q_1$ and the conditions in Theorem \ref{Sobolev estimates}. The real interpolation theory (see \cite{BeLo})  implies that the estimate \eqref{eq2.11} still holds, where $L^p,L^q$ replaced by the Lorentz spaces $L^{p,2},L^{q,2}$, respectively. 

By virtue of the continuous embedding $L^p\subset L^{p,2},L^{q,2}\subset L^q$, this gives a slightly stronger estimate than \eqref{eq2.11} which will be used in proving Theorem \ref{local H-smoothing} below. We refer to \cite[Appendix C]{MiYa1} where the real interpolation theorems and basic properties of Lorentz spaces used in the present paper have been summarized.
\end{remark}

\begin{remark}
Some further comments on Theorem \ref{Sobolev estimates} are given as follows

\begin{itemize}
 \item Besides of the estimate \eqref{eq2.10} above, there exist actually more pairs $(p,q)$  such that the following Sobolev type estimates hold:
   \begin{equation}
\label{eq2.10'}
\norm{|D|^{\alpha} u}_{L^{q}}\lesssim |z|^{\frac{n}{2m}(\frac{1}{p}-\frac{1}{q})-\frac{2m-\alpha}{2m}}\norm{((-\Delta)^m-z)u}_{L^{p}},\ z\in\C\setminus\{0\},
\end{equation}
for  $2m-n<\alpha \le 2m-2n/(n+1)$ and
\begin{equation}\label{eq2.10''}
\frac{2}{n+1}\le \frac1p-\frac1q\le \frac{2m-\alpha}{n}<1,\quad 1< p<\frac{2n}{n+1},\quad \frac{2n}{n-1}<q<\infty.
\end{equation}
Clearly, the range of $\alpha$ is exactly from the conditions \eqref{eq2.10''}. When $\alpha=0$, the estimates \eqref{eq2.10'} have been proved by \cite{KRS} for $m=1$ and \cite{HYZ} for $m\ge2$.  

When $\alpha\neq 0$, these estimates \eqref{eq2.10'} are new  for all $m\ge 1$, which are indispensable to show  the sharp Kato smoothing estimates of $e^{itH}$ with potentials shown in Section \ref{local uniform resolvent estimates} below. In particular, note that uniform Sobolev estimates with $\alpha=0$ above have played a crucial role in the unique continuity of elliptic equations and the $L^p$-multiplier estimates (see e.g. \cite{KRS}, \cite{HYZ}), as well as in recent developments on Lieb-Thirring type inequalities (see e.g. \cite{Fra1}, \cite{Fra2}, \cite{Miz1} and reference therein). Hence, the authors believe that the estimates \eqref{eq2.10}  and \eqref{eq2.10'} with $\alpha \neq 0$ would potentially have more further applications.
\vskip0.2cm

\item  The range $2m-n<\alpha \le 2m-2n/(n+1)$ in Theorem \ref{Sobolev estimates} is optimal. Indeed, on the one hand, if the inequality \eqref{eq2.10} hold for $\alpha=2m-n$ , then we only have the pair $(p,q)=(1, \infty)$ satisfying the estimate \eqref{eq2.10}, which is impossible since the classical embedding estimate $\|u\|_{L^\infty}\lesssim \||D|^n u\|_{L^1}$ does not hold (if taking $z=0$ in \eqref{eq2.10}). One the other hand, if the inequality \eqref{eq2.10} hold for $\alpha>2m-2n/(n+1)$, then we can choose some $p_0>2(n+1)/(n+3)$ such that
   $$\sup_{z\in \C\setminus[0,\infty)} \||D|^{\alpha} R_0(z)\big\|_{L^{p_0}-L^{p_0'}}<\infty,$$ which leads to the following boundary resolvent estimate as $z\rightarrow \lambda$,
   \begin{equation}\label{eq2.11'}
   \big|\<|D|^{\alpha} R_0^\pm(\lambda)f,g\>\big|\lesssim \|f\|_{L^{p_0}}\|g\|_{L^{p_0}}, \ \ f,\ g\in C_0^\infty(\R^n).
   \end{equation}
   In view of the formula \eqref{eq2.4} of $R^{\pm}(\lambda)$, if $\lambda=1$, then the  \eqref{eq2.11'} gives the Fourier restriction estimate:
   \begin{equation}\label{eq2.11''}\int_{S^{n-1}}|\hat{f}(\xi)|^2d\sigma\lesssim\big|\<|D|^{\alpha} (R_0^+(1)-R_0^-(1))f,f\>\big|\lesssim\|f\|^2_{L^{p_0}}, \end{equation}
   for $p_0>2(n+1)/(n+3)$. This contradicts with the famous Stein-Tomas theorem (see e.g. \cite{Gra} ), which says that Fourier restriction operator $$\mathfrak{R}: f\in L^p(\R^n)\mapsto \widehat{f}\mid_{S^{n-1}}\in L^2(S^{n-1}),$$ is bounded if only if $1\le p\le 2(n+1)/(n+3)$.
  \vskip0.2cm

\item If $q=p'$, since $|x|^{-m+\gamma}\in L^{n/(m-\gamma),\infty}$, we learn by weak H\"older's inequality (see \cite{BeLo}) that
\begin{equation}
\label{eq2.8''}
\big\||x|^{-m+\gamma}|D|^{\gamma} R_0(z)|D|^{\gamma}|x|^{-m+\gamma}\big\|_{L^2-L^2} \lesssim \big\||D|^\gamma R_0(z)|D|^\gamma \big\|_{L^{p,2}-L^{p',2}},
\end{equation}
if  $1/p-1/p'=2(m-\gamma)/n$. Hence for any $m-n/2<\gamma\le m-n/(n+1)$, it follows immediately from the estimate \eqref{eq2.11} and Remark \eqref{remark_1} that the uniform estimates hold:
\begin{equation}
\label{eq2.8'''}
\sup_{z\in \C\setminus[0,\infty)}\big\||x|^{-m+\gamma}|D|^{\gamma} R_0(z)|D|^{\gamma}|x|^{-m+\gamma}\big\|_{L^2-L^2}<\infty.
\end{equation}
However, we remark that the range of $\gamma$ above is not sharp.  Actually, the optimal range of $\gamma$ such that the \eqref{eq2.8'''} holds is $(m-n/2, m-1/2)$ (see e.g. Kato-Yajima \cite{KaYa},  Ruzhansky and Sugimoto \cite{RuSu} ).
\end{itemize}
\end{remark}

\subsection{The compactness of $\<x\>^{-s} R_0(z)\< x\>^{-s}$  }

In the subsection, we first prove that operators  $\<x\>^{-s} R_0(z)\< x\>^{-s}$ is compact on $L^2(\R^n)$ for any $z\in \C\setminus[0,\infty)$ for any $s>0$, then by the continuity in Theorem \ref{LIP-free case}, we can extend the compactness to the boundary value operators  $\<x\>^{-s} R^{\pm}_0(\lambda)\< x\>^{-s}$ for  sufficiently large $s$.

Let $s>0$, since $\<x\>^{-s}$ is bounded, it is enough to prove that $\<x\>^{-s} R_0(z)$ is compact for any $s>0$. Moreover, by virtue of the resolvent formula $$R_0(z)=R_0(z')-(z-z')R_0(z')R_0(z),\quad z, z'\in \C\setminus[0,\infty),$$ it hence suffices to show that $\<x\>^{-s} (1+H_0)^{-1}$ is compact on $L^2(\R^n)$ for any $s>0$. Indeed, since the bounded functions $f(x)=\<x\>^{-s}$ and $g(x)=(1+|x|^{2m})^{-1}$ both decay to $0$ as $|x|\rightarrow\infty$, it hence follows that the operator $f(X)g(D)=\<x\>^{-s}(H_0+1)^{-1}$ is a compact operator of $\mathbb B(L^2)$ (see e.g. Simon \cite[p. 160]{Simon}).  Thus, by Theorem \ref{LIP-free case} (ii) and the closeness of the family of compact operators $\mathbb B_\infty(L^2)$ in $\mathbb B(L^2)$, we immediately conclude the following result:
\vskip0.2cm
\begin{theorem}
\label{compactness}
If $s>m$, then $z\mapsto\<x\>^{-s} R_0(z)\<x\>^{-s}$ is an uniform-operator-topology continuous function on $\overline{\C^{+}}$  (also on $\overline{\C^{-}}$) with compact operator values in $\mathbb B_\infty(L^2)$, where  we set $R_0(0)=(-\Delta)^{-2m}$ and $R_0(z)=R^{\pm}_0(\lambda)$ as $z=\lambda\in \overline{\C^{\pm}}$ and $\lambda>0$.
\end{theorem}

Note that, by Theorem \ref{LIP-free case}(i), if $\lambda>0$ then we only need $s>1/2$  to show that the operators $\<x\>^{-s} R_0^\pm(\lambda)\<x\>^{-s}$ are compact. For $\lambda=0$,  the restriction $s>m$ is essentially optimal in the sense that $\<x\>^{-m} R_0(0)\<x\>^{-m}$ is not compact, although it is bounded on $L^2(\R^n)$.

\section{Kato smoothing estimates of $e^{itH}$}
\label{section 3}

In this section, we will show sharp Kato smoothing estimates of $e^{itH}$ with a general bounded decaying potential $V$. Let us first discuss the spectrum of $H=H_0+V$ in Subsection \ref{Spectrum} and then give our main results in Subsection \ref{local uniform resolvent estimates}.
\subsection{The spectrum of $H_0+V$ }
\label{Spectrum}
Let $n>2m$, $H=H_0+V$, $H_0=(-\Delta)^{m}$ and $V(x)$ be a real valued measurable function satisfying $|V(x)|\lesssim \<x\>^{-s} $ for some $s>0$.

 It is well known that $V$ is a relatively compact perturbation of $H_0$ (by using the compactness of $\<x\>^{-s}(1+H_0)^{-1}$ mentioned above). Hence, $H$ is a self-adjoint operator with the same domain as $\mathcal{D}(H_0)=\H^{2\sigma}(\R^n)$ by Kato-Rellich's theorem and its essential spectrum is the same set $[0,\infty)$ as $H_0$ by Weyl's theorem. Moreover, the spectrum located at $(-\infty, 0)$  is only discrete eigenvalues of finite multiplicity with a possible limiting point at zero.

In particular, if $s>1$ (i.e.  the potential $V$ is short-range), then the positive eigenvalues embedding in the essential spectrum $(0,\infty )$ of $H$ are also discrete and of finite multiplicity, as well as their only possible limiting point is zero point (see e.g. Agmon \cite{Agmon}). Now let us sum up some spectral results of $H$ as follows:

\begin{lemma}  \label{lemma_spectrum} Let $n>2m$, $H=(-\Delta)^{m}+V$ and $V(x)$ be a real valued measurable function satisfying $|V(x)|\le C\<x\>^{-s} $ for $s>2m$. Then the following conclusions hold:
\vskip0.2cm
{\rm (i)} The spectrum $\sigma(H)=\{\lambda_1\le \lambda_2\le \cdots\le \lambda_{N_0}<0 \}\cup [0,\infty)$, where each $\lambda_j$ is the negative discrete eigenvalue of $H$ and $N_0$ denotes the number of negative eigenvalues by counting its finite multiplicity, which satisfies  with the following bound:
\begin{equation}
\label{eq3.0}
N_0=|\{\lambda_1\le \lambda_2\le \cdots\le \lambda_{N_0}<0 \}|\le C_{n, m}\int_{\R^n}|V(x)|^{n/2m}dx<\infty.
\end{equation}
Moreover, the positive eigenvalues embedding in $(0,\infty )$ are also discrete and of finite multiplicity, as well their only possible limiting point is zero point. Finally, the singular continuous spectrum is absent.
\vskip0.2cm
{\rm (ii)} If $m=2k\ (k\in \N)$, then there exists an even $C_0^\infty(\R^n)$-potential $V$ such that $H=H_0+V$ has at least one positive eigenvalue.
\vskip0.2cm
{\rm (iii)} If the potential $V$ further satisfies with the repulsive condition $(x\cdot\nabla)V(x)\le 0$, then the point spectrum $\sigma_p(H)=\emptyset$. In particular, $H$ has no any embedding positive eigenvalues.

\end{lemma}

\begin{remark}
\label{remark 3.1}Some remarks are given as follows:

\begin{itemize}
\item In the case of $m=1$
the number estimate of negative eigenvalues \eqref{eq3.0} is well-known and usually called by {\it Cwickel-Lieb-Rozenbljum bound}. see e.g. Simon \cite[p. 674]{Simon}. For the cases $m\ge 2$, it is due to Birman and Solomyak \cite{BS}.  The statements (ii ) and (iii) have been proved in \cite[Section 7]{FSWY}. In particular, the conclusion (iii) also works for much general elliptic operator $P(D)+V$ and fractional operator $(-\Delta)^s+V$ with $s>0$.
\vskip0.2cm
\item For even $m\ge2$, there exist higher-order Schr\"{o}dinger type operators $H=(-\Delta)^m+V$ with a positive eigenvalue embedded in the continuous spectrum, even for $C_{0}^{\infty}$-potentials. Indeed, if we have a strictly positive smoothing function $\phi$ such that $\phi(x)=(-\Delta)^m\phi(x)$ for $|x|>\delta>0$, then the potential
	\begin{equation}\label{V}
		V(x)=\phi^{-1}(x)\big(\phi(x)-(-\Delta)^m\phi(x)\big).
	\end{equation}
	has compact support in $ B(0, \delta)$ and satisfies $(-\Delta)^m\phi+V\phi=\phi$ (i.e. $1$ is an eigenvalue of $H$).

 For instance, in $n=3$,  since
		\begin{equation}
			\Delta(r^{-1}e^{-r})=\Big(\frac{d^{2}}{dr^{2}}+\frac{2}{r}\frac{d}{dr}\Big)(r^{-1}e^{-r})=r^{-1}e^{-r},\ \ r=|x|>\delta>0.
		\end{equation}
so we can define a radial function $\phi(x)>0$ such that $\phi(x)=|x|^{-1}e^{-|x|}$ for $|x|>\delta>0$, which clearly satisfies $\phi(x)=(-\Delta)^m\phi(x)$ for $|x|>\delta>0$.  Thus by \eqref{V}, we construct a potential $V(x)\in C_{0}^{\infty}(\R^{3})$ such that $((-\Delta)^m+V)\phi=\phi$ for each even integer $m\ge 2$. For any other dimension $n$, we also obtain the same results if choosing the Bessel kernel function $G(|x|)$ of $(1-\Delta)^{-1}$, instead of $|x|^{-1}e^{-|x|}$ (see e.g. \cite[Section 7]{FSWY} for more details and references therein).
\end{itemize}
\end{remark}

%

\subsection{Uniform resolvent estimates and Kato smoothing estimates }
\label{local uniform resolvent estimates}
In the subsection, we will show the following local $H$-supersmoothing estimates, which imply Kato smoothing estimates of $e^{itH}$.  For the end, let us give the definition of zero resonance of $H$.
\begin{definition}
\label{resonace}
Zero is said to be a resonance of $H$ if there exists some $$0\neq \psi\in \bigcap_{\sigma>m} L^{2}_{-\sigma}(\R^n)\setminus L^2(\R^n)$$  such that $ (-\Delta)^m\psi+V\psi=0$ in the distributional sense. In particular, we say that zero is an eigenvalue of $H$ if $\psi\in L^2(\R^n)$.
\end{definition}
Note that under the assumptions on $V=O(\<x\>^{-s})$ for $s>2m$ with $n>2m$, the solution $\psi\in \bigcap_{\sigma>m} L^{2}_{-\sigma}(\R^n)$ which satisfies with $ (-\Delta)^m\psi+V\psi=0$,  must belong to the smaller space $\bigcap_{\sigma>2m-\frac{n}{2}} L^{2}_{-\sigma}(\R^n)$ if $2m<n\le 4m$ and be an eigenfunction in $L^{2}(\R^n)$ if $n>4m$. This means that zero is not a resonance and only is a possible eigenvalue if $n>4m$. In particular, as $m=1$, these situations exactly return to the case of Schr\"odinger operator $-\Delta+V$. For more details and comments, one may see Jensen and Kato \cite{JK} for $m=1$ and Feng et al. \cite{FSWY} for $m\ge 2$.
\vskip0.3cm
Let $\mathcal E_\nu=\{\lambda\in \R\ |\ \mathrm{\mathop{dist}}_{\,\R}(\lambda,\sigma_p(H))<\nu\}$ be the $\nu$-neighborhood of $\sigma_p(H)$ in $\R$. If $H$ has no eigenvalues, we set $\mathcal E_\nu=\emptyset$. Note that if $H$ has no nonnegative eigenvalues nor a zero resonance, then $\sigma_p(H)$ consists of finitely many negative discrete eigenvalues due to Lemma \ref{lemma_spectrum} (i).

\begin{theorem}
\label{local H-smoothing}
Let $n>2m$, $H=(-\Delta)^m+V$ and $|V(x)|\le C \<x\>^{-s}$ for some $s>2m$. Assume that $H$ has no positive eigenvalues and no zero resonance/eigenvalue. Then, for any $\nu>0$,
\begin{equation}
\label{eq3.1'}
\sup_{\lambda\in\R\setminus\mathcal E_{\nu},0<\theta\le1}\Big\||x|^{-m+\gamma}|D|^{\gamma}\ (H-\lambda\mp i\theta)^{-1}\ |D|^{\gamma}|x |^{-m+\gamma}\Big\|_{L^2-L^2}\le C_\nu<\infty,
\end{equation}
for any $m-{n\over 2}<\gamma< m-\frac{1}{2}$.  Moreover,  if $\gamma=m-\frac{1}{2}$, then for any $\epsilon>0$
\begin{equation}
\label{eq3.1}
\sup_{\lambda\in\R\setminus\mathcal E_{\nu},0<\theta\le1}\Big\|\<x\>^{-1/2-\epsilon}|D|^{m-1/2}\ (H-\lambda\mp i\theta)^{-1}\ |D|^{m-1/2}\<x\>^{-1/2-\epsilon}\Big\|_{L^2-L^2}\le C_\nu<\infty,
\end{equation}

\end{theorem}
\begin{proof} The proof are divided into the following two steps.
\vskip0.2cm
{\it Step 1}.  Let $M(z)=I+wR_0(z)v$ for $z\in \C\setminus[0,\infty)$, where $v(x)=\sqrt{|V|}$ and $w(x)=v(x)\sgn V(x)$.  In the sequel, we will show that the inverse $M^{-1}(\lambda\pm i\theta)$ exists  on $L^2(\R^n)$ for $(\lambda,\theta)\in(\R\setminus\mathcal E_{\nu})\times[0,1]$ and that $M^{-1}(\lambda\pm i\theta)$ is continuous on $(\R\setminus\mathcal E_{\nu})\times [0,1]$, satisfying
\begin{equation}
\label{eq3.2}\sup_{\lambda\in\R\setminus\mathcal E_{\nu},\,0\le\theta\le1}\big\| M^{-1}(\lambda\pm i\theta)\big\|_{L^2-L^2}<\infty.
\end{equation}
Clearly $V(x)=v(x)w(x)$ and $|v(x)|=|w(x)|\le C\<x\>^{-s/2}$. In view of the compactness and continuity of $\<x\>^{-s/2} R_0(z)\<x\>^{-s/2}$ on $\overline{\C^{+}}$  ( also on $\overline{\C^{-}}$ ) from Theorem \ref{compactness}, it follows that the operator $wR_0(z)v$ is also compact on $L^2(\R^n)$ for each $z\in \C\setminus[0,\infty)$. In particular, as an operator valued function, $z\mapsto wR_0(z)v$ is analytic in $\C^{\pm}$ and extends continuously to the boundary set $\R$ in the uniform topology of $\mathbb B(L^2)$. Note that by the high energy estimate \eqref{highenergy}, we have that $\|wR_0(z)v\|_{L^2-L^2}\le 1/2 $ if  $z\in \C\setminus[0,\infty)$ and $|z|\ge r$ for some $r>0$. Hence by Neumann series expansion, it follows that the series
$$M^{-1}(z)=\sum_{j=0}^{\infty}(-1)^j (wR_0(z)v)^j$$
converges uniformly in the operator norm for $|z|\ge r$ and $z\in \C\setminus[0,\infty)$. Thus by the continuity in $z$ of $wR_0(z)v$, we have that  $M^{-1}(\lambda\pm i\theta)$ is continuous on $[r,\infty)\times [0,1]$ and  \begin{equation} \label{eq3.2'}
\sup_{\lambda\ge r,\, 0\le\theta\le1}\|M^{-1}(\lambda\pm i\theta)\|_{L^2-L^2}\le 2.
\end{equation}

It remains to deal with the invertibility and continuity of $M(z)$ on the domain $\Omega_\pm:=((-\infty,r]\setminus\mathcal E_{\nu})\pm i[0,1]$. Let $z\in \Omega_\pm$. Since $wR_0(z)v$ is compact on $L^2(\R^n)$, it follows by Fredholm-Riesz theory that the inverse $M^{-1}(z)$ exists in $\mathbb B(L^2)$ if and only if $\Ker_{L^2}(M(z))=\{0\}$. To show $\Ker_{L^2}(M(z))=\{0\}$, we suppose there exists $\psi\in L^2$ such that
\begin{equation} \label{eq3.3}
M(z)\psi=\psi+wR_0(z)v\psi=0.
 \end{equation}
 Set $f=v\psi\in L^2_{s/2}(\R^n)$ with some $s>2m$ and $g=R_0(z)f$. Then we have  $(H_0+V-z)g=0$. In the sequel, we will divide the three cases to show $\psi=0$.
\vskip0.2cm
Case (i). If $z\in \Omega_\pm$ and $\Im z>0$ or $z=\lambda<0$, then $g=R_0(z)f\in L^2(\R^n)$ and $g$ must be 0 due to $z\notin\sigma(H)$. So $f=v\psi=0$, which implies that $\psi=0$ from the  equation \eqref{eq3.3}.
\vskip0.2cm
  Case (ii). If $\Omega_\pm \ni z=\lambda>0$, then $g=R_0^{+}(\lambda)f\in L^2_{-\beta}(\R^n)$ for $\beta\ge s>1/2$ by Theorem \ref{LIP-free case}(i) and $(H_0+V-\lambda)g=0$. In fact, by Agmon-H\"ormander scattering theory (see e.g. H\"ormander \cite[Theorem 14.5.2]{Ho}), we can obtain that $g$ is a rapidly decreasing eigenfunction, i.e.
 \begin{equation} \label{eq3.4}
 \int_{\R^n}(1+|x|^2)^N|(D^\alpha g)(x)|^2 dx<\infty\quad\text{for all $N\in \N$ and  $|\alpha|\le 2m$}.
 \end{equation}
Note that \eqref{eq3.4} also holds for all eigenfunctions associated with negative eigenvalues of $H$.
This means $\lambda>0$ must be an eigenvalues of $H$, which will contract our assumption unless $g=0$.  Thus as shown in the case (i), we can deduce $\psi=0$.
\vskip0.2cm
 Case (iii). If $z=0$, then $g=R_0(0)f\in \bigcap_{\sigma>m}L^2_{-\sigma}(\R^n)$  by Theorem \ref{LIP-free case}(ii) and $(H_0+V)g=0$. Since zero is not neither resonance nor zero eigenvalue of $H$ from the assumption, so $g=0$, which again deduce $\psi=0$.
\vskip0.2cm
Now we need to prove the inverse operator function $M^{-1}(z)$ is continuous for $z\in\Omega_\pm$. We may consider the case $z\in \Omega_+$ only. In fact,  let $z_0\in \Omega_+$, since $wR_0(z)v$ is continuous on $z$ in $\mathbb B(L^2)$, hence for any $\epsilon>0$, there exists a $\delta>0$ depending on $z_0$ such that when $|z-z_0|<\delta$,
\begin{equation*} \label{eq3.5}
\big\|M^{-1}(z)-M^{-1}(z_0)\big\|_{L^2-L^2}
=
\Big\|\Big(\big(I+\big(wR_0(z)v-wR_0(z_0)v\big)M^{-1}(z_0)\big)^{-1}-I\Big)M^{-1}(z_0)\Big\|_{L^2-L^2}<\epsilon.
\end{equation*}
Thus the continuity of $M^{-1}(z)$ can give
$\sup_{(\lambda,\theta)\in \Omega_+}\big\| M^{-1}(\lambda\pm i\theta)\big\|_{L^2-L^2}<\infty$, which combining with the \eqref{eq3.2'}, leads to the desired \eqref{eq3.2}.
\vskip0.3cm
{\it Step 2}. Let $z=\lambda\pm i\theta\in (\R\setminus \mathcal E_\nu)\pm i(0,1]$. Firstly, recall that the following resolvent formula
\begin{equation}
\label{resolvent formula}
R(z)=(H_0+V-z)^{-1}=R_0(z)-R_0(z)wM^{-1}(z)vR_0(z)
\end{equation}
holds, where $M(z)=I+wR_0(z)v$ and $V(x)=v(x)w(x)$ defined as in Step 1 above. To prove \eqref{eq3.1}, in view of the inequality \eqref{eq2.5} of $R_0(z)$,  it suffices to show the following estimate:
\begin{equation}
\label{eq3.6}
\sup_{\lambda\in\R\setminus \mathcal E_\nu,0<\theta\le1}\Big\|\<x\>^{-1/2-\epsilon}|D|^{m-1/2}\ R_0(z)wM^{-1}(z)vR_0(z)\ |D|^{m-1/2}\<x\>^{-1/2-\epsilon}\Big\|_{L^2-L^2}<\infty,
\end{equation}
for any $\epsilon>0$. Since $\<x\>^{-1/2-\epsilon}|D|^{m-1/2} R_0(z)w(x)$ is essentially dual each other with the operator $v(x)R_0(z) |D|^{m-1/2}\<x\>^{-1/2-\epsilon}$, using of the uniform estimate \eqref{eq3.2} for $M^{-1}$(z), it suffices to show
\begin{equation}
\label{eq3.7}
\sup_{z\in \C\setminus[0,\infty)}\bignorm{\<x\>^{-1/2-\epsilon}|D|^{m-1/2}\ R_0(z)w}_{L^2-L^2}<\infty.
\end{equation}
We now apply Theorem \ref{Sobolev estimates} to show \eqref{eq3.7} as follows. Choose $(1/p_0,1/q_0)=((n+2m)^-/2n),$ $(n-1)^+/2n)$ be such that the inequality \eqref{eq2.11} holds, where $a^{+}$( resp. $a^{-}$ ) denotes some number which is arbitrarily close but larger (resp. less) than $a$.  Then by H\"older inequality  we have
 \begin{equation}
\label{eq3.9}
\Big\|\<x\>^{-\frac{1}{2}-\epsilon}|D|^{m-\frac{1}{2}}\ R_0(z)w\Big\|_{L^2-L^2}\le \|\<x\>^{-\frac{1}{2}-\epsilon}\|_{L^{(2n)^-}}\big\||D|^{m-\frac{1}{2}}\ R_0(z)\big\|_{L^{p_0}-L^{q_0}}\|w\|_{L^{(n/m)^-}},
\end{equation}
where we have used the estimate $|w(x)|\lesssim \<x\>^{-m-}$ and $\<x\>^{-m-}\in L^{(n/m)^-}$.
Hence we get \eqref{eq3.7} and then the desired  estimate \eqref{eq3.1}.

To prove the \eqref{eq3.1'}, recall that the same uniform estimate for the free resolvent
\begin{equation}
\label{eq3.10}
\sup_{z\in \C\setminus[0,\infty)}\big\||x|^{-m+\gamma}|D|^{\gamma} R_0(z)|D|^{\gamma}|x|^{-m+\gamma}\big\|_{L^2-L^2}<\infty
\end{equation}
holds for any $m-n/2<\gamma<m-1/2$ (see \cite{RuSu}). Then in view of the resolvent formula \eqref{resolvent formula} and a similar argument above, it suffices to prove
\begin{equation}
\label{eq3.11}
\sup_{z\in \C\setminus[0,\infty)}\big\||x|^{-m+\gamma}|D|^{\gamma}\ R_0(z)w\big\|_{L^2-L^2}<\infty.
\end{equation}
To this end, we note that, as mentioned in Remark \ref{remark_1}, \eqref{eq2.11} together with the real interpolation theory implies
\begin{equation}
\label{eq3.12}\sup_{z\in \C\setminus[0,\infty)}\big\||D|^{\gamma} R_0(z)\big\|_{L^{p_0,2}-L^{q_0,2}}<\infty,\end{equation}
where  $(1/p_0,1/q_0)=((n+2m)/2n, (n-2m+2\gamma)/2n )$. Hence combining with weak H\"older's inequality, we obtain from \eqref{eq3.12} that
  \begin{equation}
\label{eq3.13}\big\||x|^{-m+\gamma}|D|^{\gamma} R_0(z)w\big\|_{L^2-L^2}\le \||x|^{-m+\gamma}\|_{L^{{n\over m-\gamma},\infty}}\big\||D|^{\gamma} R_0(z)\big\|_{L^{p_0,2}-L^{q_0,2}}\|w\|_{L^{\frac nm,\infty}}\lesssim 1
\end{equation}
uniformly in $z\in \C\setminus[0,\infty)$, which immediately deduce the desired bound \eqref{eq3.11}.
\end{proof}

If $H$ has no eigenvalues (in which case $\mathcal E_\nu=\emptyset$), then Theorem \ref{local H-smoothing} means that $|x|^{-m+\gamma}|D|^{\gamma}$ and $\<x\>^{-1/2-\epsilon}|D|^{m-1/2}$ are $H$-supersmooth in the sense of Kato-Yajima \cite{KaYa}, which implies Kato smoothing estimates, i.e. Theorem \ref{Kato-smoothing}. However, if $H$ has eigenvalues, then Theorem \ref{local H-smoothing} is not sufficient to achieve Theorem \ref{Kato-smoothing} since we have to deal with not only $e^{itH}$,  but also its absolutely continuous part $e^{itH}P_{ac}(H)$. To this end, we need to replace the resolvent ($H-z)^{-1}$ by its absolutely continuous part in Theorem \ref{local H-smoothing}, which is the main point of the following corollary.

\begin{corollary}
\label{corollary_3}
Let $P_{ac}(H)$ denote the projection onto the absolutely continuous spectral space of $H=(-\Delta)^m+V$. Then, under the conditions in Theorem \ref{local H-smoothing}, $|x|^{-m+\gamma}|D|^{\gamma}P_{ac}(H)$ and $\<x\>^{-1/2-\epsilon}|D|^{m-1/2}P_{ac}(H)$ are $H$-supersmooth, i.e. the following uniform estimates hold:

\begin{align*}
\sup_{\lambda\in \R\,\theta\in (0,1]}\big\||x|^{-m+\gamma}|D|^\gamma P_{ac}(H)(H-\lambda\mp i\theta)^{-1}P_{ac}(H)|D|^\gamma |x|^{-m+\gamma}\big\|_{L^2-L^2}&<\infty,
\end{align*}
and
\begin{align*}
\sup_{\lambda\in \R\,\theta\in (0,1]}\big\|\<x\>^{-1/2-\epsilon}|D|^{m-1/2}P_{ac}(H)(H-\lambda\mp i\theta)^{-1}P_{ac}(H)|D|^{m-1/2}\<x\>^{-1/2-\epsilon}\big\|_{L^2-L^2}&<\infty.
\end{align*}
\end{corollary}

To prove this corollary, we prepare two lemmas. The first one, which is a special case of \cite[Theorem B$^*$]{StWe}, concerns with the weighted $L^2$-boundedness of the fractional integral operator.

\begin{lemma}[{\cite[Theorem B$^*$]{StWe}}]
\label{lemma_fractional}
Let $0<\lambda<n$, $\alpha,\beta<n/2$, $\alpha+\beta\ge0$ and $\lambda+\alpha+\beta=n$. Then $|x|^{-\beta}|D|^{-n+\lambda} |x|^{-\alpha}$ extends to a bounded operator on $L^2(\R^n)$.
\end{lemma}

The second one is concerned with several mapping properties of the projection $P_{ac}(H)$. Note that the estimates \eqref{eq_lemma_projection_1} and \eqref{eq_lemma_projection_2} are sufficient to obtain Corollary \ref{corollary_3}, while \eqref{eq_lemma_projection_3} will be used to prove Strichartz estimates (i.e. Theorem \ref{theorem_4_1}) in the next section.

\begin{lemma}
\label{lemma_projection}
Let $n>2m$, $m\in \N$, $0\le \gamma<m-1/2$ and $H=(-\Delta)^m+V$ be as in Theorem \ref{local H-smoothing}. Then
\begin{align}
\label{eq_lemma_projection_1}
\||x|^{-m+\gamma}|D|^{\gamma}P_{ac}(H)f\|_{L^2}&\lesssim \||x|^{-m+\gamma}|D|^{\gamma}f\|_{L^2};
\end{align}
\begin{align}
\label{eq_lemma_projection_2}
\|\<x\>^{-1/2-\epsilon}|D|^{m-1/2}P_{ac}(H)f\|_{L^2}&\lesssim \|\<x\>^{-1/2-\epsilon}|D|^{m-1/2}f\|_{L^2},\quad \epsilon>0.
\end{align}
Moreover, for any admissible pair $(p,q)$ satisfying \eqref{admissible} with $\alpha=n/2$,
\begin{align}
\label{eq_lemma_projection_3}
\||D|^{-2(m-1)/p}P_{ac}(H)f\|_{L^{q'}}\lesssim \||D|^{-2(m-1)/p}f\|_{L^{q'}}.
\end{align}
\end{lemma}

\begin{proof}
Under the assumption of $H$, we know that $H$ has at most finitely many negative eigenvalues and there are neither embedded eigenvalues nor singular continuous spectrum (see Lemma \ref{lemma_spectrum}). Hence, with some finite integer $N_0\ge0$, $P_{ac}(H)$ is of the form $$P_{ac}(H)f=f-\sum_{j=1}^{N_0}\<\psi_j,f\>\psi_j,$$
where $\psi_1,...,\psi_{N_0}$ are eigenfunctions associated with the negative eigenvalues of $H$.

Let $G:=|x|^{-m+\gamma}|D|^{\gamma}$ and $(G^{-1})^*=|x|^{m-\gamma}|D|^{-\gamma}$. Then one has
\begin{align}
\label{projection estimate}
\norm{GP_{ac}(H)f}_{L^2}\le \|Gf\|_{L^2}+\sum_{j=1}^{N_0}\|G\psi_j\|_{L^2}\|(G^{-1})^*\psi_j\|_{L^2}\|Gf\|_{L^2}.
\end{align}
Note that Lemma \ref{lemma_fractional} with $(\lambda,\alpha,\beta)=(n-m+\gamma,0,m-\gamma)$ and \eqref{eq3.4}  imply
$$
\|G\psi_j\|_{L^2}=\||x|^{-m+\gamma}|D|^{-m+\gamma}\cdot |D|^m\psi_j\|_{L^2}\lesssim \||D|^m\psi_j\|_{L^2}<\infty.
$$
Again, using Lemma \ref{lemma_fractional} with $(\lambda,\alpha,\beta)=(n-\gamma,m,-m+\gamma)$  and \eqref{eq3.4}, we also have
$$
\|(G^{-1})^*\psi_j\|_{L^2}=\||x|^{m-\gamma}|D|^{-\gamma}|x|^{-m}\cdot |x|^m\psi_j\|_{L^2}\lesssim \||x|^m\psi_j\|_{L^2}<\infty.
$$
Therefore, the desired estimate \eqref{eq_lemma_projection_1} follows from \eqref{projection estimate}. The proof of \eqref{eq_lemma_projection_2} is essentially the same as \eqref{eq_lemma_projection_1} by using  the variant of Lemma \ref{lemma_fractional} and we thus omit it.

Next, let us consider the last estimate \eqref{eq_lemma_projection_3}. By the same argument, it suffices to show that $|D|^{2(m-1)/p}\psi_j\in L^{q}$ and $|D|^{-2(m-1)/p}\psi_j\in L^{q'}$, which can be also  deduced from \eqref{eq3.4} as follows. Firstly, since $p,q$ satisfy $2/p=n(1/2-1/q)$, the Sobolev inequality implies
$$
\||D|^{2(m-1)/p}\psi_j\|_{L^q}\lesssim \| |D|^{2/p}|D|^{2(m-1)/p}\psi_j\|_{L^2}=\||D|^{2m/p}\psi_j\|_{L^2}\le \|\<D\>^{m}\psi_j\|_{L^2}<\infty,
$$
where we have used the fact $p\ge2$ and \eqref{eq3.4}. Secondly, by \eqref{eq3.4} with $N>n/2$ and H\"older's inequality, we have $\psi_j\in L^1\cap L^2$ and hence $\psi_j\in L^r$ for any $1\le r\le2$. This fact, combined with Sobolev's inequality, shows that
$$
\||D|^{-2(m-1)/p}\psi_j\|_{L^{q'}}\lesssim \|\psi_j\|_{L^r}<\infty,
$$
where $1/r=1/q'+2(m-1)/(np)=1/2+2m/(np)>1/2$. Therefore, we have proved the desired estimate \eqref{eq_lemma_projection_3}.
\end{proof}

\begin{proof}[\underline{Proof of Corollary \ref{corollary_3}}]
We only prove the result for $G:=|x|^{-m+\gamma}|D|^{\gamma}$, since the proof for the operator $\<x\>^{-1/2-\epsilon}|D|^{m-1/2}$ is analogous. Note that $P_{ac}(H)=P_{ac}(H)^2$, it is enough to show
\begin{align}
\label{eq3.14}
\sup_{\lambda\in \R,\,\theta\in (0,1]}\norm{GP_{ac}(H)(H-\lambda\mp i\theta)^{-1}G^*}_{L^2-L^2}<\infty.
\end{align}
The proof of \eqref{eq3.14} is divided into two cases $0\le \gamma<m-1/2$ and $m-n/2<\gamma<0$.
\vskip0.2cm
We first let $0\le \gamma<m-1/2$ and observe from Lemma \ref{lemma_projection} that $P_{ac}(H)$ satisfies
\begin{align}
\label{eq3.15}
\norm{GP_{ac}(H)f}_{L^2}\lesssim \norm{Gf}_{L^2},\quad f\in D(G).
\end{align}
Let $\nu>0$ be so small that $ \mathrm{\mathop{dist}}_{\,\R}(\mathcal E_\nu,[0,\infty))\ge \nu$ which is possible since $\sigma_p(H)$ consists of finitely many negative eigenvalues under our assumption (see Lemma \ref{lemma_spectrum} above). We consider two cases $\lambda\in \R\setminus\mathcal E_\nu$ or $\lambda\in \mathcal E_\nu$ separately. If $\lambda\in \R\setminus\mathcal E_\nu$, \eqref{eq3.14} follows from \eqref{eq3.1'} and \eqref{eq3.15}. If $\lambda\in \mathcal E_\nu$, then since $\|Gf\|_{L^2}\lesssim \||D|^mf\|_{L^2}$ by Hardy's inequality ( also as seen in the proof of Lemma \ref{lemma_projection} ), we have
$$
\norm{GP_{ac}(H)(H-\lambda\mp i\theta)^{-1}G^*}_{L^2-L^2}\lesssim \||D|^mP_{ac}(H)(H-\lambda\mp i\theta)^{-1}|D|^m\|_{L^2-L^2}.
$$
Since $\||D|^m\<H\>^{-1/2}\|_{L^2-L^2}\lesssim 1$ by the fact $D(H)=\mathcal H^{2m}(\R^n)$, the spectral theorem implies
\begin{align*}
\||D|^mP_{ac}(H)(H-\lambda\mp i\theta)^{-1}|D|^m\|_{L^2-L^2}
&\lesssim \|P_{ac}(H)\<H\>(H-\lambda\mp i\theta)^{-1}\|_{L^2-L^2}\\
&=\sup_{t\in [0,\infty)}(\<t\>|t-\lambda\mp i\theta|^{-1})
\lesssim \nu^{-1}
\end{align*}
where we have used the fact $ \mathrm{\mathop{dist}}_{\,\R}(\lambda,[0,\infty))\ge \nu$. This proves \eqref{eq3.15} for $0\le \gamma<m-1/2$.

Next, the case $m-n/2<\gamma<0$ follows easily from the previous case. Indeed, by letting
$
\lambda=n+\gamma,\ \alpha=-m$ and $\beta=m-\gamma
$
in Lemma \ref{lemma_fractional}, one has $|x|^{-m+\gamma}|D|^\gamma|x|^m\in \mathbb B(L^2)$. We also have $|x|^m|D|^{\gamma}|x|^{-m+\gamma}\in \mathbb B(L^2)$ by taking the adjoint. Therefore, \eqref{eq3.14} with $\gamma=0$ implies that
\begin{align*}
&\||x|^{-m+\gamma}|D|^\gamma P_{ac}(H)(H-\lambda\mp i\theta)^{-1}|D|^\gamma |x|^{-m+\gamma}f\|_{L^2}\\
&\lesssim \||x|^{-m}P_{ac}(H)(H-\lambda\mp i\theta)^{-1}|x|^{-m}\|_{L^2-L^2}\||x|^m|D|^{\gamma}|x|^{-m+\gamma}f\|_{L^2}\\
&\lesssim \|f\|_{L^2}
\end{align*}
for $m-n/2<\gamma<0$ uniformly in $\lambda\in \R$ and $\theta\in (0,1]$. This completes the proof of \eqref{eq3.14}.
\end{proof}

Now we are in a position to give the proof of Theorems \ref{Kato-smoothing} and \ref{theorem_2}, which actually are the direct consequences of the following theorem.
\begin{theorem}
\label{Kato-smoothing_2}
Let $n>2m$, $H=(-\Delta)^m+V$ and $|V(x)|\le C \<x\>^{-s}$ for $s>2m$. Assume that $H$ has no positive eigenvalues and no zero resonance/eigenvalue. Then the following statements ( i.e. Kato smoothing estimates ) were proved:
\vskip0.2cm
 (i) If $m-n/2< \gamma< m-1/2$,  then
\begin{align*}
\big\||x|^{-m+\gamma}|D|^\gamma e^{itH}P_{ac}(H)\psi_0\big\|_{L^2_tL^2_x}&\lesssim \norm{\psi_0}_{L^2_x},\\
\left\||x|^{-m+\gamma}|D|^\gamma \int_0^t e^{-i(t-s)H}P_{ac}(H)F(s)ds\right\|_{L^2_tL^2_x}&\lesssim \norm{|x|^{m-\gamma}|D|^{-\gamma}F}_{L^2_tL^2_x}.
\end{align*}
In particular, as $\gamma=0$, the following local decay estimate holds:
\begin{align*}
\big\||x|^{-m}e^{itH}P_{ac}(H)\psi_0\big\|_{L^2_tL^2_x}&\lesssim \norm{\psi_0}_{L^2_x},\\
\left\||x|^{-m} \int_0^t e^{-i(t-s)H}P_{ac}(H)F(s)ds\right\|_{L^2_tL^2_x}&\lesssim \norm{|x|^{m}F}_{L^2_tL^2_x}.
\end{align*}
(ii) If $\gamma=m-1/2$, then for any $\epsilon>0$,
\begin{align*}
\big\|\<x\>^{-1/2-\epsilon} |D|^{m-1/2}e^{itH}P_{ac}(H)\psi_0\big\|_{L^2_tL^2_x}&\lesssim \norm{\psi_0}_{L^2_x},\\
\left\|\<x\>^{-1/2-\epsilon} |D|^{m-1/2}\int_0^t e^{-i(t-s)H}P_{ac}(H)F(s)ds\right\|_{L^2_tL^2_x}&\lesssim \norm{\<x\>^{1/2+\epsilon} |D|^{-m+1/2}F}_{L^2_tL^2_x}.
\end{align*}
\end{theorem}

\begin{proof}
These estimates are direct consequences of Corollary \ref{corollary_3} ( i.e.  the $H$-supersmoothness of $|x|^{-m+\gamma}|D|^{-\gamma}P_{ac}(H)$ and $\<x\>^{-1/2-\epsilon}|D|^{m-1/2}P_{ac}(H)$ ) and Kato's smooth perturbation theory \cite{Kat} (see also \cite{DAn} for the inhomogeneous estimates).
\end{proof}

Finally, note that  if $(x\cdot\nabla) V(x)\le 0$ and $\lim\limits_{x\to\infty}V(x)=0$, then $V(x)\ge0$. In fact, it can be easily concluded by the following integral
$$V(x)=-\int^\infty_1\frac{d}{ds}\big(V(sx)\big)ds\ge0, \ \ x\neq0,$$
where $\frac{d}{ds}\big(V(sx)\big)=\frac{1}{s}(sx\cdot \nabla)V(sx)\le 0$.  Thus, we  obtain  under these conditions that $H=H_0+V$ is a nonnegative self-adjoint operator and that $P_{ac}(H)=\Id$ in the previous results.

\section{Strichartz estimates of $e^{itH}$}
 \label{section 4}
In this section, we prove Theorem \ref{theorem_4_1}. The proof relies on a method by Rodnianski-Schlag \cite{RoSc} (see also \cite{BPST2} for the homogenous endpoint estimate and \cite{BoMi} for the double endpoint estimate). This method requires the corresponding estimates for the free evolutions which are summarized the following lemma.

\begin{lemma}
\label{lemma_4_2}
Let $H_0=(-\Delta)^m$ with $m\in \N$, $(p_1,q_1),(p_2,q_2)$ satisfy \eqref{admissible} with $\alpha=n/2$. Then,
\begin{align}
\label{eq4.5}
\||D|^{2(m-1)/p_1}e^{itH_0}\psi_0\|_{L^{p_1}_tL^{q_1}_x}&\lesssim \|\psi_0\|_{L^2_x},\\
\label{eq4.6}
\left\||D|^{2(m-1)/p_1}\int_0^te^{i(t-s)H_0}F(s)ds\right\|_{L^{p_1}_tL^{q_1}_x}&\lesssim \||D|^{2(1-m)/p_2}F\|_{L^{p_2'}_tL^{q_2'}_x}.
\end{align}
\end{lemma}
\begin{proof}
The lemma follows from dispersive estimates for $|D|^{n(m-1)}e^{itH_0}$ and Keel-Tao's theorem \cite{KeTa}. We refer to \cite[Section 3]{Pau} and \cite[Appendix A]{MiYa1} for details.
\end{proof}

\begin{proof}[\underline{Proof of Theorem \ref{theorem_4_1}}]
Note that \eqref{eq4.4} follows from \eqref{eq4.3} and Sobolev's inequality since
$$
\frac{1}{q_1}-\frac{1}{q}=\left(\frac{2m}{n}-\frac{2}{n}\right)\frac{1}{p_1}=\frac{2(m-1)}{p_1n}
$$
as long as $(p_1,q_1)$ is $n/2$-admissible and $(p_1,q)$ is $n/(2m)$-admissible. It is thus enough to show \eqref{eq4.3}. Let $\Lambda_p=|D|^{2(m-1)/p}$ for short. For a given self-adjoint operator $A$, we set
$$
\Gamma_AF(t,x)=\int_0^t e^{i(t-s)A}F(s,x)ds.
$$
Then \eqref{eq4.3} can be decomposed into the following two estimates
\begin{align}
\label{eq4.7}
\|\Lambda_{p_1}e^{itH}P_{ac}(H)\psi_0\|_{L^{p_1}_tL^{q_1}_x}&\lesssim \|\psi_0\|_{L^2_x},\\
\label{eq4.8}
\|\Lambda_{p_1}\Gamma_HP_{ac}(H)F\|_{L^{p_1}_tL^{q_1}_x}&\lesssim \|\Lambda_{p_2}^{-1}F\|_{L^{p_2'}_tL^{q_2'}_x}.
\end{align}

We first consider \eqref{eq4.8} whose proof relies on the following Duhamel formulas:
\begin{align*}
\Gamma_H=\Gamma_{H_0}-i\Gamma_{H_0}V\Gamma_H=\Gamma_{H_0}-i\Gamma_{H}V\Gamma_{H_0}
\end{align*}
(see e.g. \cite[Section 4]{BoMi} for the proof of these formulas). By \eqref{eq4.6}, one first has
\begin{align}
\nonumber
\|\Lambda_{p_1}\Gamma_HP_{ac}(H)F\|_{L^{p_1}_tL^{q_1}_x}
&\le \|\Lambda_{p_1}\Gamma_{H_0}P_{ac}(H)F\|_{L^{p_1}_tL^{q_1}_x}
+\|\Lambda_{p_1}\Gamma_{H_0}V\Gamma_HP_{ac}(H)F\|_{L^{p_1}_tL^{q_1}_x}\\
\label{eq4.9}
&\lesssim \|\Lambda_{p_2}^{-1}P_{ac}(H)F\|_{L^{p'_2}_tL^{q_2'}_x}+\|\Lambda_{p_1}\Gamma_{H_0}VP_{ac}(H)\Gamma_HF\|_{L^{p_1}_tL^{q_1}_x}.
\end{align}
Since Lemma \ref{lemma_projection} implies
$$
\|\Lambda_{p_2}^{-1}P_{ac}(H)F\|_{L^{q_2'}_x}\lesssim \|\Lambda_{p_2}^{-1}F\|_{L^{q_2'}_x},
$$
hence the term $\Lambda_{p_2}^{-1}P_{ac}(H)F$ in \eqref{eq4.9} satisfies the desired estimate
\begin{align}
\label{eq4.10}
\|\Lambda_{p_2}^{-1}P_{ac}(H)F\|_{L^{p'_2}_tL^{q_2'}_x}\lesssim \|\Lambda_{p_2}^{-1}F\|_{L^{p_2'}_tL^{q_2'}_x}.
\end{align}
For the operator $\Lambda_{p_1}\Gamma_{H_0}VP_{ac}(H)\Gamma_HF$, we first apply \eqref{eq4.6} with $(p_2,q_2)=(2,\frac{2n}{n-2})$ to obtain
\begin{align}
\label{eq4.11}
\|\Lambda_{p_1}\Gamma_{H_0}VP_{ac}(H)\Gamma_HF\|_{L^{p_1}_tL^{q_1}_x}\lesssim \|\Lambda_2^{-1}VP_{ac}(H)\Gamma_HF\|_{L^2_tL^{\frac{2n}{n+2}}_x}.
\end{align}
Recall here that $V=vw$ with $v,w\in L^{n/m}$. With the equalities
$$
\frac{n+2m}{2n}-\frac{n+2}{2n}=\frac{m-1}{n},\quad \frac{n+2m}{2n}=\frac mn+\frac12
$$
at hand, we see from Sobolev's and H\"older's inequalities that
\begin{align}
\label{eq4.12}
\|\Lambda_2^{-1}VP_{ac}(H)\Gamma_HF\|_{L^2_tL^{\frac{2n}{n+2}}_x}\lesssim \|v\|_{L^\frac nm_x}\|wP_{ac}(H)\Gamma_HF\|_{L^2_tL^2_x}.
\end{align}
Now we use again the Duhamel formula to estimate the right hand side of \eqref{eq4.12} as
\begin{align}
\label{eq4.13}
\|wP_{ac}(H)\Gamma_HF\|_{L^2_tL^2_x}\le \|wP_{ac}(H)\Gamma_{H_0}F\|_{L^2_tL^2_x}+\|wP_{ac}(H)\Gamma_{H}V\Gamma_{H_0}F\|_{L^2_tL^2_x}.
\end{align}
Since $\|wP_{ac}(H)f\|_{L^2}\lesssim \|wf\|_{L^2}$ (which can be verified by the same proof as that of \eqref{eq_lemma_projection_1}), the first term of the right hand side of \eqref{eq4.13} is dominated by $\|w\Gamma_{H_0}F\|_{L^2_tL^2_x}$. Moreover, since $|x|^{-m}P_{ac}(H)$ is $H$-supersmooth by Corollary \ref{corollary_3}, so is $wP_{ac}(H)$ and hence $wP_{ac}(H)\Gamma_{H}w\in \mathbb B(L^2_tL^2_x)$ by the same argument as in the proof of Theorem \ref{Kato-smoothing_2}. Therefore, the second term of the right hand side of \eqref{eq4.13} is dominated by $\|v\Gamma_{H_0}F\|_{L^2_tL^2_x}$. Since $|v|=|w|$, we conclude that
\begin{align}
\nonumber
\|wP_{ac}(H)\Gamma_HF\|_{L^2_tL^2_x}
&\lesssim \|w\Gamma_{H_0}F\|_{L^2_tL^2_x}\\
\nonumber
&\lesssim \|w\|_{L^{\frac nm}_x}\|\Lambda_2\Gamma_{H_0}F\|_{L^2_tL^{\frac{2n}{n-2}}_x}\\
\label{eq4.14}
&\lesssim \|F\|_{L^{p_2'}_tL^{q_2'}_x},
\end{align}
where we have used H\"older's and Sobolev's inequalities in the second line and \eqref{eq4.6} with $(p_1,q_1)=(2,\frac{2n}{n-2})$ in the last line. Finally, \eqref{eq4.9}--\eqref{eq4.14} yield the desired bound \eqref{eq4.8}.

Next, the estimate \eqref{eq4.7} can be obtained similarly by using \eqref{eq1.2}, \eqref{eq4.5}, \eqref{eq4.6} and the usual Duhamel formula $U_H=U_{H_0}-i\Gamma_{H_0}VU_H$, where $U_{H_0}=e^{itH_0}$ and $U_H=e^{itH}$. Since the proof is essentially same as (or even simpler than) that of \eqref{eq4.8}, we omit these details.
\end{proof}

\section{The proofs of Theorems \ref{LIP-free case} and \ref{Sobolev estimates}  }

\label{section 5}
\subsection{The proof of Theorem \ref{LIP-free case}} The statements (i) and (iii) in  Theorems \ref{LIP-free case}  are due to Agmon \cite{Agmon} (also see Remark \ref{remark 2.1} for more comments). Hence, it remains to show the statement (ii) of Theorems \ref{LIP-free case}. Actually, comparing with the conclusion (i) of Theorems \ref{LIP-free case}, it suffices to add the continuity of the resolvent $R_0(z)$ at  $z=0 $ in the uniform operator topology of $\mathbb B(L^2_s, L^2_{-s})$ if $s>m$.  
Following the arguments of Agmon \cite[Theorem 4.1]{Agmon}, we first prove  the weak continuity of $R_0(z)$ on $\overline{\C^{\pm}}$, and then lift to the uniform-topology continuity of $R_0(z)$ by the compact arguments.

\begin{proof}[\underline{The proof of Theorem \ref{LIP-free case}(ii)}:]  The proof is divided into the following three steps:
\vskip0.2cm
{\it Step 1}. Let $s>m$ and $f,g\in L^2_{s}$. We first prove that  $\<R_0(z)f, g\>$, which is an analytic function on $\C\setminus[0,\infty)$, has a continuous boundary value on the two sides of  $[0,\infty)$. Clearly, for any $\lambda>0$, by the statement (i) of  Theorem \ref{LIP-free case}, we have
\begin{equation}
\label {eq5.1}
\lim_{\C^\pm \ni z\rightarrow \lambda}\<R_0(z)f, g\>=\<R^\pm_0(\lambda)f, g\>,
\end{equation}
where $R^\pm_0(\lambda)=R_0(\lambda\pm i0)$ are defined in \eqref{eq2.4}. For $\lambda=0$, we need to show the limit
\begin{equation}
\label {eq5.2}\lim_{\C^\pm \ni z\rightarrow 0}\<R_0(z)f, g\>=\<R_0(0)f, g\>
 \end{equation} holds, where $R_0(0)=(-\Delta)^{-m}$. Note that $|\<R_0(z)f, g\>|\le C \|f\|_{L^2_s}\|g\|_{L^2_s}$ holds uniformly in $z\in \C\setminus[0,\infty)$.  So we may assume $f,\ g\in S(\R^n)$ (the Schwartz function class) by a density argument. By the kernel function expansion of $(-\Delta-z)^{-1}$ and the decomposition formula
\begin{equation}
R_{0}(z)=\frac{1}{mz}\sum_{\ell=0}^{m-1}z_\ell\big(-\Delta-z_\ell\big)^{-1}, \ z_\ell=z^{\frac{1}{m}}e^{i\frac{2\ell\pi}{m}},\ z\in \C\setminus[0,\infty),
\end{equation}
 we can establish the kernel expansion of $R_0(z)$ as $z$ close to $0$:
$$R_0(z)(x,y)=c_n|x-y|^{2m-n}+E(z,x,y),$$
where $c_n|x-y|^{2m-n}$ is the kernel of $R_0(0)$ and the integral operator $E(z)$ with the kernel $E(z,x,y)$ belongs to $\mathbb B(L^2_\sigma,L^2_{-\sigma})$ for any $\sigma>n/2+2$ and satisfies $\|E(z)\|_{L^2_\sigma-L^2_{-\sigma}}=O(|z|^\epsilon)$ for some $\epsilon>0$ depending on $n, m$ (see e.g. \cite[Proposition  2.4]{FSWY}). Hence it follows that
$$|\<(R_{0}(z)-R_0(0))f, g\>|\le \|E(z)\|_{L^2_\sigma-L^2_{-\sigma}} \|f\|_{L^2_\sigma}\|g\|_{L^2_\sigma}\rightarrow 0,$$
as $z\rightarrow0$  for any $f,g\in S(\R^n)$,  from  which we thus conclude \eqref{eq5.2} by a density argument.

\vskip0.2cm
{\it Step 2}. Let $\H^k_\sigma(\R^n)$ denote the weighted Sobolev space defined by the norm $$\|f\|_{\H^k_\sigma(\R^n)}=\|(1-\Delta)^{k/2}f\|_{L^2_\sigma(\R^n)}.$$  Also define $R^\pm_0(z)=R^\pm(\lambda)$ if $z=\lambda>0$, and $R^\pm_0(z)=R(z)$ if $z\in \overline{\C^\pm}\setminus(0,\infty)$. By Step 1, we know that the operator function $R^\pm{(z)}$ is weak continuous on $\overline{\C^\pm}$.  In this step, we will show that $R^\pm{(z)}$ is continuous on $\overline{\C^\pm} $ in the strong operator topology of $\mathbb B(L^2_s, L^2_{-s})$ for $s>m$.  First, note that the equality $$(1+H_0)R_0(z)f=f+(z+1)R_0(z)f$$ hold for $z\in \C\setminus[0,\infty)$, which leads to the following uniform bounds for any $M>0$.
\begin{equation}
\label{eq5.3}
\|R_0(z)f\|_{\H^{2m}_{-s}(\R^n)}\le C_M \|f\|_{L^2_s(\R^n)}, \ z\in \C\setminus[0,\infty), \ |z|\le M.
\end{equation}
By the weak continuity of $R^\pm_0(z)$ on $\mathbb B(L^2_s, L^2_{-s})$ and \eqref{eq5.3}, we can actually obtain for any $f\in L^2_s$ and any $z_0\in \overline{\C^\pm}$ that
\begin{equation}\label{eq5.4}
\wlim_{\C^\pm \ni z\rightarrow \lambda} R_0^\pm(z)f=R_0^\pm(z_0 )f,\  {\rm in}\  \H^{2m}_{-s}(\R^n), \end{equation}
and $R^{\pm}_0(z)\in B(L^2_s, \H^{2m}_{-s})$ for any $z\in \overline{\C^\pm}$ and $ s>m$.  Since $R^\pm_0(z) f\in\H^{2m}_{-s'} $ for any $f\in L^2_s(\R^n)$ and the  space embedding relation $\H^{2m}_{-s'}\hookrightarrow L^2_{-s}(\R^n)$ is compact for any $0<s'<s$, hence by the compactness, we can lift the weak convergence of \eqref{eq5.4} up to the following strong convergence:
\begin{equation}\label{eq5.5}
\slim_{\C^\pm \ni z\rightarrow z_0} R_0^\pm(z)f=R_0^\pm(z_0 )f,\  {\rm in}\  L^2_{-s}(\R^n),
\end{equation}
where $f\in L^2_s(\R^n)$. Thus we conclude that $R^\pm{(z)}$ is  continuous on $\overline{\C^\pm} $ in the strong operator topology of $\mathbb B(L^2_s, L^2_{-s})$ for $s>m$.
\vskip0.2cm
{\it  Step 3.} Finally, we come to prove that $R^\pm{(z)}$ is  continuous on $\overline{\C^\pm} $ in the uniform operator topology of $\mathbb B(L^2_s, L^2_{-s})$ for $s>m$. To the end, suppose that it is not true by contradiction. Then there exist sequences $\{z_j\}\subset \overline{\C^\pm}$ with $z_j\rightarrow z_0\in \overline{\C^\pm}$, and $\{f_j\}\subset L^2_s(\R^n)$ with $\|f_j\|_{L^2_s}=1$, such that
\begin{equation}\label{eq5.6}\limi_{j\to\infty}\|(R^\pm_0(z_j)-R^\pm_0(z_0))f_j\|_{L^2_{-s}}>0.\end{equation}
Note that $\{f_j\}$ always has a weak convergent subsequence,  so we may assume that $f_j\rightharpoonup f$ in $L^2_{s}(\R^n)$ (in weak sense). Now by using a similar argument in Step 2, we actually can prove that the following convergence holds:
 \begin{equation}\label{eq5.7}
\slim_{j\rightarrow \infty} R_0^\pm(z_j)f_j=R_0^\pm(z_0 )f,\  {\rm in}\  L^2_{-s}(\R^n),
\end{equation}
which clearly gives a contradiction to \eqref{eq5.6}. Thus summing up three Steps above, we have finished the proof of Theorem \ref{LIP-free case} (ii).
\end{proof}

\subsection{The proof of Theorem \ref{Sobolev estimates}}

Theorem \ref{Sobolev estimates} is actually the special case (i.e. $1/p-1/q=(2m-\alpha)/n$) of the estimates \eqref{eq2.3-fract} in the following lemma, which will be proved based on the Fourier method involved with the famous Carleson-Sj\"olin oscillatory integral argument (see \cite[p.69]{So}). Note that the argument of the proof have been used similarly in Sikora-Yan-Yao \cite{SYY}. Here we emphasize that the following results are new as $\alpha\neq 0$ even for the second order case,  and crucial to Kato smoothing estimates studies in this paper.

\begin{lemma}\label{thm2.111}  Let $n>2m$, $H_0=(-\Delta)^m$ and $z\in \mathbb{C}$. Consider
arbitrary auxiliary cutoff function $\psi$ such that
$\psi\in C_0^{\infty}({\mathbb R}),
 \psi(s)\equiv 1$ if $s\in [1/2, 2]$ and  $ \psi$ is supported in the interval  $[1/4, 4]$.
Suppose  also that
 exponent $(1/p, 1/q)\in (0,1)^2$ satisfy the following conditions:
\begin{eqnarray}\label{eq2.4-fract}\min\Big(\frac{1}{p}-\frac{1}{2},\
\frac{1}{2}-\frac{1}{q}\Big)> {1\over 2n},\ \ \frac{2}{n+1}\le\Big(\frac{1}{p}-\frac{1}{q}\Big)\le 1.
\end{eqnarray}
%
Then there exists  positive constants $C_{p,q,\alpha}$ independent of $|z|$ such that
\begin{eqnarray}\label{eq2.3-psi}
\|\left|D|^{\alpha}(H_0-z\right)^{-1}\psi(H_0/|z|) \|_{p\to q}\le
C_{p,q,\alpha}\ |z|^{\frac{n}{2m}(\frac{1}{p}-\frac{1}{q})-\frac{2m-\alpha}{2m}},
\end{eqnarray}
for all $z\in \mathbb{C}^{\pm}\setminus\{0\}$  and $\alpha\in \R$, where $|D|=\sqrt{-\Delta}$. Furthermore, besides of the conditions \eqref{eq2.4-fract},  assume  that $\frac{1}{p}-\frac{1}{q}
\le \frac{2m-\alpha}{n}$  and $2m-n<\alpha \le 2m-{2n\over n+1}$. Then, for all $z\in \mathbb{C}^{\pm}\setminus\{0\}$,
\begin{eqnarray}\label{eq2.3-fract}
\|\left|D|^{\alpha}(H_0-z\right)^{-1}\|_{p\to q}\le C_{p,q,\alpha}\ |z|^{\frac{n}{2m}(\frac{1}{p}-\frac{1}{q})-\frac{2m-\alpha}{2m}}.
\end{eqnarray}
\end{lemma}

\begin{proof}
By a scaling argument, we may assume $z=e^{i\theta}$ with
$0<|\theta|\le \pi$. If $\delta<|\theta|\le\pi$ for any small $\delta>0$,
  then $|D|^\alpha (H_0-e^{i\theta})^{-1}$ is  a standard constant coefficient Fourier multiplier
  operator of order $-2m+\alpha$ with the symbol $|\xi|^\alpha(|\xi|^{2m}-e^{i\theta})^{-1}$.
  Hence the estimate  \eqref{eq2.3-fract}
  follows from the standard Sobolev estimates.
 A similar argument shows that for any $p \le q$ the multiplier $|D|^\alpha(H_0-e^{i\theta})^{-1}\psi(H_0)$
 is bounded as as operator from $L^p$ to $L^q$.
Thus we may assume that $0<|\theta|\le \delta$ (that is, $z$ belongs to cone neighborhood containing positive real line),  and by symmetry it is enough  to consider only the case $\Im z>0$.

Using the reduction above, we  may set $z=(\lambda+i\lambda\varepsilon)^{2m}$ for some $\lambda\sim 1$
 and $0<\varepsilon\ll1$. Since $|z|\sim\lambda^{2m}\sim 1$, by a scaling argument again in $\lambda$,
 it suffices to estimate $|D|^\alpha(H_0-(1+i\varepsilon)^{2m})^{-1
 }$     and $|D|^\alpha(H_0-(1+i\varepsilon)^{2m})^{-1
 }\psi(H_0)$
uniformly  for  $0<\varepsilon\ll1$.
Let $K^\varepsilon$ be the convolution kernel of $|D|^\alpha(H_0-(1+i\varepsilon))^{-1 }$.  Then Fourier transform gives that
$$K^\varepsilon=\mathcal{F}^{-1}\Big(|\xi|^\alpha \big(|\xi|^{2m}-(1+i\varepsilon)^{2m}\big)^{-1}\Big).$$
Decompose $K^\varepsilon=K_1+K_2$, where
$$K_1=\mathcal{F}^{-1}\Big(\frac{|\xi|^\alpha\psi(|\xi|^{2m})}{|\xi|^{2m}-(1+i\varepsilon)^{2m}}\Big),\quad K_2=\mathcal{F}^{-1}\Big( \frac{|\xi|^\alpha(1-\psi(|\xi|^{2m}))}{|\xi|^{2m}-(1+i\varepsilon)^{2m}}\Big).$$
To show \eqref{eq2.3-psi}
and \eqref{eq2.3-fract},  it is crucial to verify  that the operator $K_1*f$ satisfies \eqref{eq2.3-psi},
since \eqref{eq2.3-fract} can immediately follows by combining $K_1*f$ with the simpler part $K_2*f$.

\vskip0.2cm

{\it Estimate for  $K_2*f$.}  By the support property of $\psi$ the symbol of $K_2$ satisfies that
 $$\Big|D^\beta\Big( \frac{|\xi|^\alpha(1-\psi(|\xi|^{2m}))}{|\xi|^{2m}-(1+i\varepsilon)^{2m}}\Big)\Big|\le
 C_{\alpha\beta}(1+|\xi|)^{-2m+\alpha-|\beta|}, \ \ \xi\neq 0,$$
for any $\beta\in \N_0^n$. Hence by the Fourier transform  we obtain that $|K_2(x)|\le C_N |x|^{2m-\alpha-n-N}$ for any $N\in \mathbb{N}_0$. Then
  Young's inequality and interpolation (note that $2m-\alpha<n$) give that
 \begin{equation}\label{eq2.7-fract}\|K_2*f\|_{L^q}\le C_{p,q}\|f\|_{L^p}
  \end{equation}
  for all $(p,q)$ satisfying $0\le\frac{1}{p}-\frac{1}{q}\le {2m-\alpha\over n}$ and $1<p\le q<\infty$.
\vskip0.2cm
{\it Estimate for  $K_1*f$.} In order to apply the stationary phase method to $K_1$, we first write
  \begin{equation}\label{eq2.8-fract}K_1(x)=\int _{\R ^n}\frac{e^{ix\xi}\
  \widetilde{\psi}(|\xi|)}{|\xi|^{2m}-(1+i\varepsilon)^{2m}}\ d\xi=
  \int_0^\infty\frac{s^{n-1}\widetilde{\psi}(s)}{s-1-i\varepsilon}\
  \Big(\int_{S^{n-1}}e^{isx\omega}d\omega\Big) ds,\end{equation}
where $\xi=s\omega$, $\widetilde{\psi}(s)=s^\alpha \psi(s^{2m})(s^{2m-1}+s^{2m-2}(1+i\varepsilon)+\ldots+(1+i\varepsilon)^{2m-1})^{-1}$.


Note that $K_1$ is the Fourier transform of a compactly supported distribution including taking
limits with $\varepsilon $ goes to $\pm0$ so
$|K_1(x)|\le C$  for  all $|x|\le 1$.  To handle the remaining
case $|x|>1$, we recall the following stationary phase formula for the Fourier
transform of a smooth measure on hypersurface~$S^{n-1}$(see e.g. \cite[p.51]{So}):
\begin{equation}\label{surface-fract}\int_{S^{n-1}}e^{iy\omega}d\omega
=|y|^{-\frac{n-1}2}c_+(y)e^{i|y|}+|y|^{-\frac{n-1}2}c_-(y)e^{-i|y|},
\end{equation}
where, for $|y|\ge 1/4$, the coefficients satisfy
\begin{equation}\label{surface estmates-fract}
\Bigl|\frac{\partial^\beta}{\partial y^\beta}c_+(y)\Bigr|
+\Bigl|\frac{\partial^\beta}{\partial y^\beta}c_-(y)\Bigr|
\le C_\alpha|y|^{-|\beta|}, \quad \beta\in \mathbb{N}_0.
\end{equation}
Thus combining \eqref{eq2.8-fract} with \eqref{surface-fract}, one has
 \begin{eqnarray}\label{eq2.11-fract}
 K_1(x)&=&\sum_{\pm}\int_0^\infty\frac{s^{n-1}\widetilde{ \psi}(s)}{s-1-i\varepsilon}\
 \Big(|sx|^{-\frac{n-1}2}c_\pm(sx)e^{\pm is|x|}\Big) ds \nonumber\\
& =& \sum_{\pm}|x|^{-\frac{n-1}{2}}b_\varepsilon^\pm (x) e^{\pm i|x|},\ \ |x|>1/4,
 \end{eqnarray}
where
$$ b_\varepsilon^\pm(x)=\int_{-\infty}^\infty\frac{(s+1)^{\frac{n-1}{2}}
\widetilde{ \psi}(s+1)}{s-i\varepsilon} c_\pm((s+1)x)e^{\pm is|x|} ds.$$
Note that  the  function $s\mapsto (s+1)^{\frac{n-1}{2}}\widetilde{ \psi}(s+1)c_\pm((s+1)x)$ is smooth and compactly supported near $s=0$.
So one can obtain  uniformly in $\varepsilon>0$ that
$$|\partial^\beta b_\varepsilon^\pm(x)|\le C_\beta |x|^{-|\beta|}, \ \ |x|>1/4.
$$
Hence in view of \eqref{eq2.11-fract}, we can further smoothly decompose $K_1=K'+K''$
  in  such a way that $\supp K'\subset B(0, 1)$ (the unit ball of $\R^n$), $K'$ is bounded
   and  $K''$  can  be expressed as
   \begin{eqnarray}\label{eq2.11'-fract}
  K''(x)=\sum_{\pm}|x|^{-\frac{n-1}{2}}a_\pm(x) e^{\pm i|x|},
   \end{eqnarray}
where $a_\pm \in C^\infty(\R ^n)$ satisfy $a_{\pm}(x)=0$ for $|x|\le 1/2$ and
$|\partial^\beta a_\pm(x)|\le C_\beta |x|^{-|\beta|}$ for any $\beta\in \mathbb{N}_0$.
By Young's inequality, we have for all $1\le p \le q \le \infty$ that
\begin{equation}\label{eq2.12-fract}
\|K'*f\|_{L^q}\le C\|f\|_{L^p}
\end{equation}

To estimate  $K''$, we first note that  $|K''(x)|\le (1+|x|)^{-(n-1)/2}$ from the expression \eqref{eq2.11'-fract}.
Hence, for all $1<p\le q<\infty$ satisfying $\frac{n+1}{2n}\le\frac{1}{p}-\frac{1}{q}\le1$, one has
\begin{equation}\label{eq2.13-fract}
\|K''*f\|_{L^q}\le C\|f\|_{L^p}.
\end{equation}
However, this argument does not
give the whole range
of pairs $(p,q)$ for which  \eqref{eq2.13-fract} holds. It is possible to  extend it by using the
oscillatory factor $e^{\pm i|x|}$  of $K''(x)$ in  the formula \eqref{eq2.11'-fract}, i.e.
\begin{equation}\label{eq2.13'-fract}
K''*f(x)=\sum_{\pm}\int_{\R^ n}|x-y|^{-\frac{n-1}{2}}a_\pm(x-y) e^{\pm i|x-y|}f(y)dy.
\end{equation}
Indeed, the phase function $|x-y|$ satisfies the so-called $n\times n$-Carleson-Sj\"olin
conditions, see \cite[p.69]{So}. Hence the celebrated  Carleson-Sj\"olin
argument  can be used to estimate $K''*f$.
\vskip0.2cm
Let $\phi(s)\in C_0^\infty(\R )$ be a such function that supp~$\phi\in [\frac{1}{2},2]$ and
$\sum_{\ell=0}^\infty\phi(2^{-\ell}s)=1$ for $s\ge {1/2}$. Set $K''_\ell(x)=
\phi(2^{-\ell}|x|)K''(x)$ for all $\ell=0,1,2,\ldots$, so
\begin{equation}\label{eq2.13''-fract}
 K''*f(x)=\sum_{\ell=0}^\infty (K_\ell''*f)(x),
 \end{equation}
where
\begin{equation}\label{eq2.14'-fract}
K_\ell''*f(x):=\int_{\R^ n}|x-y|^{-\frac{n-1}{2}}\phi(2^{-\ell}|x-y|)a_\pm((x-y)) e^{\pm i|x-y|}f(y)dy.
\end{equation}
Put $\lambda=2^{\ell}$. Then the scaling gives
\begin{equation}\label{eq2.15-fract}(K_\ell''*f)(\lambda x)=\lambda^{\frac{n+1}{2}}\int_{\R ^n} w(x-y)e^{\pm \lambda i|x-y|}f(\lambda y)dy,
\end{equation}
where
$
w(x)=|x|^{-\frac{n-1}{2}}\phi(|x|)a_\pm(\lambda x))\in C_0^\infty(\R^n \setminus 0)$
satisfying  $|\partial^\beta w(x)|\le C_\beta$ for any $\beta\in \N_0^n$.  Now we can apply  Carleson-S\"ojlin argument
(see \cite[p.69]{So}) to \eqref{eq2.15-fract}, obtaining that
$$ \|K_\ell''*f\|_q\le C \lambda^{-n/p+(n+1)/2}\|f\|_{L^p},\ \ \lambda=2^\ell, \ \ell=0,1,\ldots,$$
and hence
\begin{equation}\label{eq2.14-fract}
\|K''*f\|_{L^q}\le C\|f\|_{L^p},
\end{equation}
for all $q=\frac{n+1}{n-1}p'$, $1\le p<2n/(n+1)$  as $n\ge3$. Furthermore,
 by interpolating between the estimates
 \eqref{eq2.13-fract} and \eqref{eq2.14-fract},  we can conclude that
 \begin{equation}\label{eq2.15'-fract}
\|K''*f\|_{L^q}\le C\|f\|_{L^p},
\end{equation}
 for all $(p,q)$ such that $\frac{2}{n+1}< \frac{1}{p}-\frac{1}{q}\le1$ and
 $$\min\Big(\frac{1}{p}-\frac{1}{2},\ \frac{1}{2}-\frac{1}{q}\Big)> {1\over 2n}.$$
 Therefore the estimates \eqref{eq2.7-fract}, \eqref{eq2.12-fract} together with \eqref{eq2.15'-fract} yield the estimate
 \eqref{eq2.3-fract} besides of the  boundary line $2/(n+1)=1/p-1/q$. But based on the oscillatory integral presentations \eqref{eq2.13'-fract} and \eqref{eq2.15-fract} of $K''$, this can be proved by showing the weak estimates of the endpoint case, and using duality and real interpolation method.  We refer reader to see  \cite{Gut} or \cite{HZ} for such technical details of the remained cases.  Thus we have finished the proof of Lemma~\ref{thm2.111}.
\end{proof}

\section*{Acknowledgments}
H. Mizutani is partially supported by JSPS KAKENHI Grant-in-Aid for Scientific Research (C) \#JP21K03325.  X. Yao are partially supported by NSFC grants No.11771165 and 12171182. Finally, we also thank the anonymous referee for helpful comments which greatly improve the present version.


\end{document}